\documentclass[]{amsart}
\usepackage{amsmath,amsfonts,amssymb,amsthm}
\usepackage{enumerate}
\usepackage{mathrsfs}
\usepackage[usenames]{color}
\DeclareMathOperator{\Exp}{exp}

\DeclareMathOperator{\codim}{codim}
\DeclareMathOperator{\Diff}{Diff}
\DeclareMathOperator{\Real}{Re}
\DeclareMathOperator{\Imag}{Im}

\DeclareMathOperator{\diag}{diag}

\DeclareMathOperator{\Spec}{spec}
\DeclareMathOperator{\Sing}{Sing}
\newcommand{\xx}{{\bf x}}
\newcommand{\yy}{{\bf y}}
\newcommand{\ww}{{\bf w}}
\newcommand{\zz}{{\bf z}}
\newcommand{\diff}[2]{\mbox{{\rm Diff}{${\,}_{#1}({\mathbb C}^{#2},0)$}}}
\newcommand{\Diffh}{\widehat{\Diff}}
\newcommand{\petal}{R_{d,e,\varepsilon}}
\newcommand{\basin}{S^m_{d,e,\varepsilon}}

\newcommand{\spaceH}{\mathcal{H}^m_{d,e,\varepsilon}}

\newcommand{\cvf}{\hat{\mathfrak{X}}}
\def\N{\mathbb N}

\def\R{\mathbb R}
\def\C{\mathbb C}

\def\Cd#1{(\C^#1,0)}

\def\mc#1{{\mathcal #1}}

\newcommand{\vp}{\varphi}
\def\G{\Gamma}
\def\g{\gamma}
\newcommand{\la}{{\lambda}}
\newcommand{\wt}[1]{{\widetilde{#1}}}

\newcommand{\ol}[1]{\overline{#1}}

\newcommand{\parcial}[1]{\frac{\partial}{\partial {#1}}}

\theoremstyle{plain}
\newtheorem{theorem}{Theorem}[section]
\newtheorem{proposition}[theorem]{Proposition}
\newtheorem{lemma}[theorem]{Lemma}
\newtheorem{cor}[theorem]{Corollary}
\theoremstyle{definition}
\newtheorem{definition}[theorem]{Definition}
\newtheorem{remark}[theorem]{Remark}
\theoremstyle{plain}
\newtheorem{Theorem}{Theorem}
\newtheorem{Proposition}[Theorem]{Proposition}
\theoremstyle{definition}

\newcommand{\obra}[3]{{\sc #1} {\em #2}. {#3}.}

\newcommand{\diffh}[2]{\mbox{$\widehat{\rm Diff}{{\,}_{#1}({\mathbb C}^{#2},0)}$}}

\newcommand{\cn}[1]{\mbox{(${\mathbb C}^{#1},0$)}}

\author[L. L\'opez-Hernanz]{Lorena L\'opez-Hernanz}
\address{Lorena L\'opez-Hernanz\\ Departamento de F\'isica y Matem\'{a}ticas\\
Universidad de Alcal\'a\\ Spain} \email{lorena.lopezh@uah.es}

\author[J. Rib\'on]{Javier Rib\'on}
\address{Javier Rib\'on\\ Instituto de Matem\'{a}tica e Estat\'istica\\Universidade Federal Fluminense\\ Brazil}
\email{javier@mat.uff.br}

\author[F. Sanz S\'anchez]{Fernando Sanz S\'anchez}
\address{Fernando Sanz S\'anchez\\ Departamento de \'{A}lgebra, An\'{a}lisis Matem\'{a}tico, Geometr\'{\i}a y Topolog\'{\i}a\\ Universidad de Valladolid\\ Spain} \email{fsanz@agt.uva.es}

\author[L. Vivas]{Liz Vivas}
\address{Liz Vivas\\ Department of Mathematics\\The Ohio State University\\USA} \email{vivas@math.osu.edu}

\thanks{First, second and third authors partially supported by Ministerio de Ciencia e Innovaci\'on, Spain, process MTM2016-77642-C2-1-P and PID2019-105621GB-I00. Fourth author partially supported by NSF-1800777.}
\date{}
\title[Stable manifolds asymptotic to formal curves]
{Stable manifolds of biholomorphisms in $\C^n$ asymptotic to formal curves}
\begin{document}
\begin{abstract}
Given a germ of biholomorphism $F\in\Diff\Cd n$ with a formal invariant curve $\G$ such that the multiplier of the restricted formal diffeomorphism $F|_\G$ is a root of unity or satisfies $|(F|_\G)'(0)|<1$, we prove that either $\G$ is contained in the set of periodic points of $F$ or there exists a finite family of stable manifolds of $F$ where all the orbits are asymptotic to $\G$ and whose union eventually contains every orbit asymptotic to $\G$.
This result generalizes to the case where $\Gamma$ is a formal  periodic curve.
\end{abstract}
\maketitle
\section{Introduction}\label{sec:intro}

In this paper, we generalize to arbitrary dimension $n$ the main result proved by L\'{o}pez-Hernanz et al. in \cite{Lop-R-R-S} for $n=2$.
Namely, we consider a germ of biholomorphism $F\in\Diff\Cd n$ with a formal invariant curve $\G$ at the origin.  Assuming that the multiplier $\lambda$ of the restricted (formal) diffeomorphism $F|_\G$ is a root of unity or satisfies $|\lambda|<1$, we prove that either $\G$ is contained in the set of periodic points of $F$ or there exists a finite family of stable manifolds of $F$ whose union consists of and contains eventually any orbit of $F$ asymptotic to $\Gamma$, i.e. having flat contact with 
$\Gamma$. Note that the condition on $\lambda$ corresponds to the necessary condition for the existence of stable orbits of the one-dimensional dynamics of $F|_\G$ when $\G$ is convergent (see P\'{e}rez-Marco \cite{Per,Per1}). 
It does not depend on the rest of multipliers of $F$.

In the two-dimensional case, the stable manifolds obtained in \cite{Lop-R-R-S} are either one-dimensional (saddle behavior) or open sets (node behavior). In dimension $n \geq 3$, we also obtain stable manifolds of intermediate dimension $1 < s < n$ that share some properties with both the saddle and the node cases of dimension two. Although the theorem is very similar,  the proof is not a straightforward generalization of the two-dimensional one. We need to introduce new techniques that we discuss below in this introduction.

Let us first describe more precisely the statement of the main theorem.

A {\em stable set} of $F$ is a subset $B\subset V$ of an open neighborhood $V$ of $0$ where $F$ is defined which is invariant, i.e. $F(B)\subset B$, and such that the orbit of each point of $B$ converges to $0$. If $B$ is an analytic, locally closed submanifold of $V$ then we say that $B$ is a {\em stable manifold} of $F$ (in $V$). Let us remark that, as in \cite{Lop-R-R-S}, our definition of stable manifold is more general than the classical one, since the stable manifolds considered in this paper do not contain the origin in general, even if by definition their closures contain it.

A formal (irreducible) curve $\G$  at $0\in\C^n$ is a prime ideal $\G$ of $\C[[x_1,...,x_n]]$ such that $\C[[x_1,\dots,x_n]]/\G$ has dimension 1.
We say that $\G$ is {\em invariant} by $F$ if $\G \circ F = \G$. In this case, we can consider the {\em restriction} $F|_\G$,
which is a formal diffeomorphism in one variable (see Section~\ref{sec:blow-ups} for details).
A non-trivial (positive) orbit $O$ of $F$ is {\em asymptotic} to
a formal curve $\G$ if $O$ converges to the origin and, for any finite composition $\sigma$ of blow-ups of points, the lifted orbit $\sigma^{-1}(O)$
has a limit equal to the point on the transform of $\G$ by $\sigma$.  
If this is the case then $\G$ is necessarily invariant for $F$ (see Section ~\ref{sec:blow-ups}).

Our main result is the following:

\begin{Theorem}\label{th:main}
	Consider $F\in\Diff\Cd n$ and let $\G$ be a formal invariant curve of $F$.   Assume that the multiplier
	$\lambda=(F|_{\G})'(0)$
	is a root of unity or satisfies $|\lambda|<1$. Then we have one of the following two possibilities:
	\begin{enumerate}[(i)]
		\item The curve $\G$ is contained in the set of fixed points of some non-trivial iterate of $F$, or
         \item $F|_{\G}$ is not periodic and
         there exist orbits of $F$ asymptotic to $\G$.
	\end{enumerate}
 In the latter case, in any sufficiently small open neighborhood $V$ of $0$ there exists a non-empty finite family of pairwise disjoint stable manifolds $S_1,...,S_r\subset V$ of $F$ of pure positive dimension and with finitely many connected components such that the orbit of every point in $S_1\cup\dots\cup S_r$ is asymptotic to $\G$ and such that any orbit of $F$ asymptotic to $\G$ is eventually contained in
$S_1\cup\dots\cup S_r$.
\end{Theorem}
The stable manifolds $S_1,...,S_r$
provide a base of asymptotic convergence along $\Gamma$ \`{a} la Ueda \cite{Ued1}. We can be more precise in the hyperbolic case:
\begin{Proposition}\label{pr:main2-hyperbolic}
Consider $F\in\Diff\Cd n$ and let $\G$ be an invariant formal curve of $F$.  Assume that the multiplier $\lambda=(F|_{\G})'(0)$
satisfies $|\lambda|<1$. Then $\G$ is a germ of an analytic curve at the origin and a representative of $\G$ is a stable manifold of $F$ that eventually contains any orbit of $F$ asymptotic to $\G$.
\end{Proposition}

The result can be stated more generally for a formal {\em periodic} curve $\G$ (i.e. $\G$ is invariant for some iterate $F^s$ of $F$). More precisely, we apply Theorem  \ref{th:main} to $F^s$ and $\Gamma$ in order to obtain stable manifolds for $F^s$, which induce stable manifolds for $F$ by simple arguments that can be found in \cite{Lop-R-R-S}.

It is worth mentioning that whereas a planar diffeomorphism $F\in\Diff\Cd 2$ always has a formal periodic curve (Rib\'{on }\cite{Rib}, see also Corollary~\ref{cor:per_curve}), this is no longer true for dimension $n\ge 3$ by an example of a holomorphic vector field of G\'{o}mez-Mont and Luengo \cite{Gom-L},
whose flow is treated by Abate and Tovena in \cite{Aba-T}.
As a consequence of the results in Section~\ref{sec:infinitesimal}, we will obtain in Section~\ref{sec:cond_for_inv_curves} a condition that guarantees the existence of formal invariant curves in dimension three,
inspired by a result of Cerveau and Lins Neto for vector fields \cite{Cerveau-Lins:cod2}.

Notice that in any dimension there are linear examples of biholomorphisms $F$ with an invariant axis for which the multiplier $\lambda$ either satisfies
$|\lambda|>1$ or is {\em irrationally neutral} (i.e. $|\lambda|=1$ and it is not a root of unity) and Theorem~\ref{th:main} does not hold. Thus, the hypothesis concerning $\lambda$ in Theorem~\ref{th:main} is necessary. In fact, for $n=1$, as we mentioned above, if there are positive orbits of $F$ converging to the origin then $\lambda$ satisfies the hypothesis of Theorem~\ref{th:main}.

\strut

Let us describe the structure of the proof of Theorem~\ref{th:main}.
After recalling in Section~\ref{sec:blow-ups} the main definitions and properties concerning formal curves, blow-ups and asymptotic orbits, in Section~\ref{sec:hyperbolic} we study the case where $F|_\G$ is hyperbolic attracting and we prove Proposition ~\ref{pr:main2-hyperbolic}. In this case, the result is a consequence of the classical stable manifold and Hartman-Grobman theorems for diffeomorphisms and the proof goes as in the two-dimensional case \cite{Lop-R-R-S}.

The case where the multiplier $\lambda$ is a root of unity is the core of the paper. We assume that $F|_\G$ is not periodic. One of the main ingredients in this case is a suitable normal form for the pair $(F,\G)$, that we call {\em Ramis-Sibuya form}. Namely, there exist coordinates $(x,\yy)$ at $0\in\C^n$ such that $\G$ is non-singular and transverse to $x=0$ and $F$ is written as
$$
F(x,\yy )=\left(x-x^{q+1}+  b x^{2q+1} + O(x^{2q+2}) ,
\exp\left(D(x)+x^qC\right)\yy+O(x^{q+1})\right),
$$
where $q\ge1$, $b \in{\mathbb C}$,
$D(x)$ is a diagonal polynomial matrix of degree at most $q-1$ and $C$ is a constant matrix such that $[D(x),C]=0$.
Ramis-Sibuya form is inspired by a classical result on normal forms of systems of linear ODEs with formal meromorphic coefficients due to Turrittin \cite{Tur}. Such normal forms are also used for non-linear systems by Braaksma \cite{Bra} and Ramis and Sibuya \cite{Ram-S} in order to prove multisummability of the formal solutions of the system (when the coefficients are convergent). In Section~\ref{sec:RS-vector-fields-definition} we define the analogous Ramis-Sibuya form for a pair $(X,\G)$, where $X$ is a formal vector field and $\G$ is a formal invariant curve of $X$, and in Section~\ref{sec:reduction} we prove that any pair $(X,\G)$ can be reduced to Ramis-Sibuya form by means of a finite number of blow-ups with smooth centers and ramifications, all of them adapted to $\G$. This result is essentially a consequence of Turrittin's theorem once we associate to $(X,\G)$ a system of $n-1$ formal meromorphic ODEs after some initial punctual blow-ups.

Then, concerning the reduction of a pair $(F,\G)$ to Ramis-Sibuya form, we use a result by Binyamini~\cite{Binyamini:finite} that guarantees that an adequate iterate $F^m$ of $F$ is the time-1 flow of a formal vector field $X$, a so called {\it infinitesimal generator},  which will allow to obtain a reduction of $(F^m, \G)$ to Ramis-Sibuya form from the corresponding one for $(X,\G)$. In order for this construction to work the condition
$F = \Exp (X)$ is not sufficient. We need the biholomorphism and the vector field to
share some additional geometrical properties; for
instance, the invariant curve $\G$ must be also invariant for $X$.
We devote Section~\ref{sec:infinitesimal} to showing
the geometrical nature of the correspondence between local biholomorphisms and infinitesimal generators.
Moreover, we determine whether $X$ is a geometrical infinitesimal
generator of $\Exp (X)$ (Theorem \ref{teo:car_inf_gen}). The condition depends only on
the eigenvalues of the linear part $D_0 X$ of $X$ at the origin.

Using these results, in Section~\ref{sec:RS-diffeos} we accomplish the reduction of the pair $(F^m,\G)$. The fact that we are replacing $F$ by an iterate presents no problem for the proof of Theorem~\ref{th:main}: if $S$ is a stable manifold of $F^m$ composed of orbits asymptotic to the invariant curve $\G$, then $\cup_{k=0}^{m-1} F^{k}(S)$ is also stable for $F$.
Moreover, the blow-ups and ramifications considered in the reduction preserve the property of asymptoticity of the orbits to the formal curve. Thus, for the proof of Theorem~\ref{th:main}, we may assume that $(F,\G)$ is already in Ramis-Sibuya form.

Finally, in Section~\ref{sec:stable-manifolds} we prove Theorem~\ref{th:main} when $(F,\G)$ is in Ramis-Sibuya form. The family of stable manifolds in the statement is associated to the family of the $q$ attracting directions of the restricted formal diffeomorphism $F|_\G(x)=x-x^{q+1}+O(x^{2q+1})$. In fact, if  $\ell$ is such an attracting direction, the Ramis-Sibuya form allows to separate the $\yy$-variables in two groups, depending only on the restriction of the polynomial matrix $D(x)+x^qC$ to $\ell$. Dynamically, each of these groups corresponds either to a saddle or to a node behavior along the orbits that converge to the origin tangentially to $\ell$. The dimension of the stable manifold associated to $\ell$ will then be equal to the number of node variables plus one. The proof of the existence of those stable manifolds is a generalization of the corresponding one in~\cite{Lop-R-R-S}, inspired by Hakim's construction in \cite{Hak2} (see also \cite{Ari-R} by Arizzi and Raissy). In particular, they are obtained as fixed points of an adequate continuous map. Instead of using Banach fixed point theorem as in \cite{Hak}, \cite{Hak2}, \cite{Lop-S} or \cite{Lop-R-R-S}, we use Schauder fixed point theorem, which simplifies significantly the technical computations of the proof. The lack of uniqueness in Schauder's theorem will be compensated by an easy alternative argument.

\section{Formal invariant curves, asymptotic orbits and blow-ups}\label{sec:blow-ups}

In this section we introduce the main definitions and properties concerning a formal invariant curve $\G$ of a
formal vector field $X$  
or a biholomorphism $F$ and the behavior of both $F$ and $\G$ under punctual blow-ups. The content of this section is just a generalization to higher dimension of what can already be found in \cite{Lop-R-R-S} for dimension two. We include it for the sake of completeness and to fix notations.

First, let us consider invariance by formal vector fields.
An {\em (irreducible) formal curve} at $0\in\C^n$ is a prime ideal $\G$ of the ring
$\hat{\mathcal{O}}_n=\C[[x_1,...,x_n]]$  
such that the quotient ring $\hat{\mathcal{O}}_n/\G$ has dimension one. It is determined by a formal parametrization $\g(s)=(\g_1(s),...,\g_n(s))\in (s\C[[s]])^n\setminus\{0\}$ so that $g(\g(s))\equiv 0$ if and only if $g\in\G$.
Let $X\in\cvf\Cd n$ be a singular   formal vector field.
More precisely, once we choose coordinates $\xx=(x_1,...,x_n)$, we write $X$ as
$$
X=a_1(\xx)\parcial{x_1}+a_2(\xx)\parcial{x_2}+\cdots+a_n(\xx)\parcial{x_n},
$$
where $a_j(\xx)=X(x_j)\in\hat{\mathcal{O}}_n$ satisfies $a_j(0)=0$. The {\em multiplicity} of $X$, denoted by $\nu(X)$, is the minimum of the orders of the series $a_j$, which is independent of the chosen coordinates. The {\em singular locus} of $X$ is the ideal $\Sing(X)\subset \hat{\mathcal{O}}_n$ generated by the series $a_1,...,a_n$.

Recall that a formal curve $\G$ is invariant for $X$ if $X(\G)\subset\G$.
In terms of a parametrization $\g(s)$ of $\G$, invariance is equivalent to the existence of $h_\g(s)\in\C[[s]]$ such that
\begin{equation}\label{eq:invariant-curve-for-X}
X|_{\g(s)}=(a_1(\g(s)),...,a_n(\g(s)))=h_\g(s)\g'(s).
\end{equation}
Notice that $h_\g(s)\equiv 0$ if and only if $\Sing(X)\subset\G$ and thus this property is independent of the parametrization (we say that $\G$ is {\em contained} in the singular locus of $X$). When $\G$ is invariant, we define the {\em restriction} of $X$ to $\G$ as the one-dimensional formal vector field $$
X|_\G=h_\g(s)\parcial s,
$$
where $\g$ is an irreducible parametrization of $\G$ and $h_\g(s)$ is defined by equation (\ref{eq:invariant-curve-for-X}). Actually $X|_\Gamma$ can be defined intrinsically since
$X(\Gamma) \subset \Gamma$ implies that $X$ defines a derivation
of the ring of formal functions $\hat{\mathcal{O}}_n/\Gamma$ of
$\Gamma$. The multiplier $\lambda_\G= h_{\g}' (0)\in \C$ is called the
{\em inner eigenvalue} of the pair $(X,\G)$.
The {\em tangent eigenvalue} of $(X,\G)$, denoted by $\lambda (\G)$, is the eigenvalue
of the differential $D_0 X$ corresponding to the tangent direction of $\G$.
These eigenvalues are related by $\nu \lambda_{\G} = \lambda (\G)$, where $\nu$ is the multiplicity of $\G$ at 0.

A formal curve $\G$ is {\em invariant} for $F\in\Diff\Cd n$ if $h\circ F\in\G$ for any $h\in\G$. Moreover, given a parametrization $\g(s)$ of $\G$, the invariance of $\G$ is equivalent to the existence of a series $\theta(s)\in\C[[s]]$ with $\theta(0)=0$ and $\theta'(0)\neq 0$ such that $F\circ\g(s)=\g\circ\theta(s)$. This series $\theta(s)$ can be seen as a formal diffeomorphism in one variable, i.e. $\theta(s)\in\widehat{\Diff}(\C,0)$. Its class of formal conjugacy is independent of the chosen parametrization $\g(s)$ and any representative of this class is called the {\em restriction} of $F$ to $\G$ and denoted by $F|_\G$.

If $\G$ is invariant for $F$, the multiplier $\lambda_\G=(F|_\G)'(0)\in\C^*$ does not depend on $\theta(s)$ and is called the {\em inner eigenvalue} of the pair $(F,\G)$. Notice that this number is preserved under reparametrizations.
On the other hand, the {\em tangent eigenvalue} $\lambda(\G)$ of $(F,\G)$
is the eigenvalue of  $D_0F$ corresponding to the tangent direction of $\G$. It is easy to check that
\begin{equation}
\label{equ:inner_tangent}(\lambda_\G)^\nu=\lambda(\G),
\end{equation}
where $\nu$ is the multiplicity of $\G$ at $0$.
In particular, we have $\lambda_\G=\lambda(\G)$ when $\G$ is non-singular.

\begin{definition}\label{def:restricted-Gamma}
Let $\G$ be a formal invariant curve of $F\in\Diff\Cd n$ and let $\lambda_\G$ be the inner eigenvalue. We say that $\G$ is {\em hyperbolic attracting} if $|\lambda_\G|<1$, and that $\G$ is \emph{rationally neutral} if $\lambda_\G$ is a root of unity.
\end{definition}

\strut

Consider a germ of biholomorphism $F\in\Diff\Cd n$.
Denote by $\pi:\wt{\C^n}\to\C^n$ the blow-up of $\C^n$ at the origin and by $E=\pi^{-1}(0)$ the exceptional divisor. The transformed biholomorphism $\wt{F}=\pi^{-1}\circ F\circ\pi$ extends to an injective holomorphic map in a neighborhood of $E$ in $\wt{\C^n}$ such that
$\wt{F}(E) = E$.
We have moreover that $\wt{F}|_E$ is the projectivization of the linear map $D_0F$ in the identification $E\simeq\mathbb{P}^{n-1}_\C$ and hence fixed points $p\in E$ for $\wt{F}$ correspond to invariant lines of $D_0F$. Such a point $p$ is a {\em first infinitely near fixed point} of $F$ and the germ $F_p$ of $\wt{F}$ at $p$ is the {\em transform} of $F$ at $p$. Blowing-up repeatedly, we define sequences $\{p_k\}_{k\geq 0}$ of {\em infinitely near fixed points} of $F$ and corresponding transforms $F_{p_k}$, where $p_0=0$.

A formal curve $\G$ is also determined by its sequence of {\em iterated tangents} $\{q_k\}_{k\ge0}$, defined by: $q_0=0$ and, for $k\ge1$, if $\pi_{q_{k-1}}$ is the blow-up at $q_{k-1}$, the point $q_k\in\pi_{q_{k-1}}^{-1}(q_{k-1})$ corresponds to the tangent line of the strict transform of $\G$ at $q_{k-1}$.
The formal curve $\Gamma$ is invariant for $F$ if and only if
the sequence of iterated tangents of $\G$ is a sequence of infinitely near fixed points of $F$ (see \cite{Lop-R-R-S}).

Note that the inner eigenvalue is invariant under blow-up and hence the
condition of $\G$ being hyperbolic attracting or rationally neutral is stable under blow-ups.

\strut

Given a formal curve $\G$ at $0\in\C^n$, a stable non-trivial orbit $O=\{a_k=F^k(a_0)\}$ of a diffeomorphism $F\in\Diff\Cd n$ is {\em asymptotic} to $\G$ if, being $\{q_k\}$ the sequence of iterated tangents of $\G$, the following holds: if $\pi_1:M_1\to\C^n$ is the blow-up at the origin then $\lim_{k\to\infty}\pi_1^{-1}(a_k)=q_1$; if $\pi_2:M_2\to M_1$ is the blow-up at $q_1$ then $\lim_{k\to\infty}\pi_2^{-1}\circ\pi_1^{-1}(a_k)=q_2$; and so on. Notice that if such an orbit exists then $\G$ is invariant for $F$, since in this case any iterated tangent $q_k$ of $\G$ must be an infinitely near fixed point of $F$.

We remark that our definition of asymptoticity to a formal curve $\G$ corresponds to the standard one of having $\G$ as ``asymptotic expansion''. For instance, if $\G$ is non-singular and we consider a parametrization of the form $\g(s)=(s, \textbf{h}(s))\in\C[[s]]^n$ in some coordinates $(x,\yy)\in\C\times\C^{n-1}$, with $\textbf{h}(s)=\sum_{j=1}^{\infty}\textbf{h}_{j} s^{j}$, then a non-trivial orbit $O=\{(x_k,\yy_k)\}$ satisfies  $\lim_{k\to\infty}\pi_1^{-1}(x_k,\yy_k)=q_1$ if and only if $\lim_{k \to \infty} {\yy_{k}}/x_k = {\bf h}_1$. Then we obtain 
$\lim_{k\to\infty}\pi_2^{-1}\circ\pi_1^{-1}(x_k,\yy_k)=q_2$ if and only if $\lim_{k \to \infty} \frac{\yy_{k}/x_k - {\bf h}_1}{x_k} = {\bf h}_2$ and so on.
Therefore $O$ is asymptotic to $\G$ if and only if for any $N\in\N$ we have
$$
\lim_{k \to \infty} \frac{\yy_k-J_N{\bf h}(x_k)}{x_k^{N+1}} = {\bf h}_{N+1},  
$$
where $J_N$ denotes the $N$-jet. This implies $\left\|\yy_k-J_N{\bf h}(x_k)\right\|\leq (|| {\bf h}_{N+1}|| +1)|x_k|^{N+1}$
for some $k\geq k_0 (N)$.
\section{Hyperbolic attracting case}\label{sec:hyperbolic}

In this section we prove Theorem~\ref{th:main} in the case where the formal curve is hyperbolic attracting. More precisely, we prove Proposition \ref{pr:main2-hyperbolic}
in the introduction that we recall next for convenience. 
\begin{proposition}
Consider $F\in\Diff\Cd n$ and let $\G$ be an invariant formal curve of $F$. Assume that $\G$ is hyperbolic attracting. Then $\G$ is a germ of an analytic curve at the origin and a representative of $\G$ is a stable manifold of $F$ that eventually contains any orbit of $F$ asymptotic to $\G$.
\end{proposition}

\begin{proof}
Let $\{q_k\}_{k\geq 0}$ be the sequence of iterated tangents of $\G$. Notice that it suffices to prove the statement for $F_{q_k}$ and $\G_k$ at any point $q_k$, where $F_{q_k}$ is the transform of $F$ at $q_k$ and $\G_k$ is the strict transform of $\G$ at $q_k$. 
Notice that the curve $\G$ has a Puiseux parametrization by the local parametrization theorem 
(cf. \cite[Volume II, Section D, Theorem 10]{Gunning:II}) and hence $\G$ can be desingularized by a sequence of punctual blow-ups. Thus,
we can assume that $\G$ is non-singular.
Let $\lambda=\lambda(\G)$ be the tangent eigenvalue of $(F,\G)$, which coincides with the inner eigenvalue $\lambda_\G$ since $\G$ is non-singular. Set $\Spec(D_0F)=\{\la,\mu_2,\ldots,\mu_n\}$. An easy computation shows that the eigenvalues of the linear part of $F_{q_1}$ at $q_1$ are given by $\{\la,\mu_2/\la,\ldots,\mu_n/\la\}$. Moreover, $\lambda$ is still the tangent eigenvalue of the pair $(F_{q_1},\G_1)$ since the inner eigenvalue is preserved under blow-up. Repeating this argument, it follows that, for each $k$, the eigenvalues of the linear part of $F_{q_k}$ at $q_k$ are $\{\la,\mu_2/\la^k,\ldots,\mu_n/\la^k\}$. Now, assume that $k$ is large enough so that $|\la|<1<|\mu_j|/|\la^{k}|$ for any $j=2,...,n$. Then, by the Stable Manifold Theorem, we obtain that $\G_k$ is the stable manifold of $F_{q_k}$ at $q_k$, hence an analytic curve. Moreover, using Hartman-Grobman Theorem, we have that the unique orbits of $F_{q_k}$ that converge to $q_k$ are those which are eventually contained in $\G_k$.
\end{proof}

\begin{remark} \label{rk:attract}
From the theory of one-dimensional dynamics, we have that the hyperbolic attracting case is the only one for which there is a stable set whose germ is an analytic curve at the origin (cf. \cite{Lop-R-R-S}).
\end{remark}

\section{Infinitesimal generator of a biholomorphism}\label{sec:infinitesimal}

In this section, we recover a result due to Binyamini~\cite{Binyamini:finite} that guarantees that for any local biholomorphism $F \in \Diff\Cd n$ there exists a formal vector field $X$ such that the time-$1$ flow $\Exp(X)$
of the vector field is a non-trivial iterate $F^m$ of $F$. We prove that this construction is ``geometrically significant" in the sense
that the geometrical properties of $F^m$ and $X$ are related.
For example, the fixed point set of $F^m$ coincides with the singular set of $X$. Moreover,
the invariance of analytic sets is preserved by this correspondence between diffeomorphisms
and formal vector fields; indeed,
$F^m$ preserves a germ of analytic set, or a formal analytic set, if and only if
$X$ does. In particular, if $\Gamma$ is a formal invariant curve that is periodic for
$F \in \Diff\Cd n$, we will see that $\Gamma$ is invariant by $X$.
This property will be crucial in Section~\ref{sec:RS-vector-fields} to obtain a reduction of the pair $(F^m,\G)$ to Ramis-Sibuya form from the corresponding one for $(X,\G)$.

\subsection{Preliminaries}\label{sec:setting}
The strategy to obtain a vector field $X$ such that $\Exp(X)=F^m$ for some $m$ is a generalization to the context of diffeomorphisms
of the correspondence between the connected component of the identity of a finite dimensional algebraic group and the Lie algebra of the group.
This generalization is possible since the group $\Diffh\Cd n$ of formal diffeomorphisms, despite being infinite dimensional, can
be interpreted as a projective limit of finite dimensional algebraic groups.
This approach allows to define the algebraic closure
$\overline{\langle F \rangle}$ of the group $\langle F\rangle$ generated by $F$, its connected component of
the identity and its associated Lie algebra.
The infinitesimal generators of iterates of $F$ are chosen in
this Lie algebra. 
First, we introduce these ideas; further details can be found in
\cite{JR:solvable25}, \cite{JR:arxivdl} and \cite{JR:finite}.

Consider the normal subgroup $N_{k}$ of $\Diffh\Cd n$ defined by
\[ N_k = \{ F \in \Diffh\Cd n : x_{j}\circ F - x_j \in {\mathfrak m}^{k+1} \
\forall 1 \leq j \leq n \} ,\] 
where ${\mathfrak m}$ is the maximal ideal of the ring
$\hat{\mathcal{O}}_{n} = {\mathbb C}[[x_1, \dots, x_n]]$ of formal power series.
It is the subgroup of formal diffeomorphisms that have order of contact at least $k+1$
with the identity map.
We denote by $D_k$ the group $\Diffh\Cd n/ N_k$ of $k$-jets of formal diffeomorphisms.
Given $F \in \Diffh\Cd n$, we can uniquely associate to $F$ the element 
\begin{equation}
\label{equ:comp}
\begin{array}{ccccc}
F_k & : & {\mathfrak m}/ {\mathfrak m}^{k+1} & \to & {\mathfrak m}/ {\mathfrak m}^{k+1} \\
& & f +  {\mathfrak m}^{k+1} & \mapsto & f \circ F +  {\mathfrak m}^{k+1}
\end{array}
\end{equation}
of the linear group
$\mathrm{GL} ({\mathfrak m}/ {\mathfrak m}^{k+1})$ which only depends on the class of $F$ in $D_k$.  
In this way we can interpret $D_k$ as a subgroup of $\mathrm{GL} ({\mathfrak m}/ {\mathfrak m}^{k+1})$.
Moreover, it is a (finite dimensional) algebraic matrix group since
$\{ F_k : F \in \Diffh\Cd n \}$ coincides with the group of
automorphisms of the ${\mathbb C}$-algebra ${\mathfrak m}/ {\mathfrak m}^{k+1}$
(cf. \cite[Lemma 2.1]{JR:solvable25}).
The Lie algebra $L_k$ of $\{ F_k : F \in \Diffh\Cd n \}$ is the Lie algebra of derivations of the
${\mathbb C}$-algebra ${\mathfrak m}/ {\mathfrak m}^{k+1}$ for any $k \geq 1$.
Moreover, $L_{k}$ can be identified with $\hat{\mathfrak X} ({\mathbb C}^{n},0) / K_k$
where $\cvf\Cd n$ is the complex Lie algebra of singular formal vector fields
(i.e. derivations of the $\C$-algebra ${\mathfrak m}$) and
$K_k = \{ X \in \cvf\Cd n: X ({\mathfrak m}) \subset {\mathfrak m}^{k+1} \}$.

The natural projections $\pi_{k,l}: D_k \to D_l$ and $(d \pi_{k,l})_{\mathrm{Id}}: L_k \to L_l$
when $k \geq l$ define inverse systems and
the group $\Diffh\Cd n$ (resp. the Lie algebra $\hat{\mathfrak X} ({\mathbb C}^{n},0)$)
can be identified with the projective limit of the groups $D_k$
(resp. the Lie algebras $L_k$) for $k \geq 1$. We denote by $\pi_k:\Diffh\Cd n\to D_k$ and
$d\pi_k:\cvf\Cd n\to L_k$ the natural maps that send the projective limits onto their factors.

\begin{definition} \label{def:Zariski-closure}
Given a subgroup $G$ of $\Diffh\Cd n$,
we denote by $G_k$ the Zariski-closure of $\pi_{k}(G)$ and by $G_{k,0}$ the connected component
of the identity of $G_k$ for $k\ge1$. Then we define
\[ \overline{G} = \varprojlim G_{k} =
\{ F \in \Diffh\Cd n : \pi_k(F)\in G_k \ \forall k \in {\mathbb N} \} \]
and
\[  \overline{G}_0 = \varprojlim G_{k,0} =
\{ F \in \Diffh\Cd n : \pi_k(F)\in G_{k,0} \ \forall k \in {\mathbb N} \} \]
as the \emph{Zariski-closure} (or \emph{pro-algebraic closure}) of $G$ and its \emph{connected component of the identity},  
respectively. Given a subgroup $G$ of $\Diffh\Cd n$, we say that it is {\it pro-algebraic} if
$\overline{G} = G$. We define the \emph{Lie algebra} ${\mathfrak g}$ of $\overline{G}$ as
\[ {\mathfrak g} =  \varprojlim {\mathfrak g}_{k} =
\{ X \in \hat{\mathfrak X} ({\mathbb C}^{n},0) : d\pi_k(X) \in {\mathfrak g}_k \ \forall k \in\N\}, \]
where ${\mathfrak g}_k$ is the Lie algebra of $G_k$ for $k \geq 1$.
\end{definition}
\begin{remark} \cite[Proposition 2]{JR:arxivdl}
The Lie algebra ${\mathfrak g}$ of $\overline G$ satisfies
\[ {\mathfrak g} = \{ X \in \hat{\mathfrak X} ({\mathbb C}^{n},0) : \Exp (t X) \in \overline{G} \ \
\forall t \in {\mathbb C} \} , \]
where $\Exp (tX)$ is the time-$t$ flow of $X$, i.e. the formal diffeomorphism that satisfies, for any $g\in\hat{\mathcal{O}}_n$,
$$g\circ\Exp(tX)=\sum_{j=0}^\infty\frac{t^jX^j(g)}{j!},$$
where $X^0(g)=g$ and $X^j(g)=X(X^{j-1}(g))$ for $j\ge1$.
\end{remark}

In the two following results we summarize several properties of the finite dimensional setting that generalize to the infinite
dimensional one and provide a criterion that allows to identify pro-algebraic groups of formal diffeomorphisms.
\begin{proposition}\label{pro:axu}
Let $G$ be a subgroup of $\Diffh\Cd n$. We have
\begin{enumerate}[(i)]
\item $\overline{G}_0$ is a finite index normal pro-algebraic subgroup of $\overline{G}$
\cite[Proposition 2.3 and Remark 2.9]{JR:solvable25}.
\item Any finite index subgroup of $\overline{G}$ is pro-algebraic and contains $\overline{G}_0$
\cite[Lemmas 2.3 and 2.1]{JR:finite}.
\item $\overline{G}_0$ is generated by $\Exp({\mathfrak g})$
\cite[Proposition 2]{JR:arxivdl}. In particular $\Exp(t X)$ belongs to $\overline{G}_0$
for any $t \in {\mathbb C}$ and any $X \in {\mathfrak g}$.
\end{enumerate}
\end{proposition}
\begin{proposition}[{\cite[Lemma 2.4]{JR:solvable25}}]
\label{pro:criterium}
Assume that  $H_k$ is an algebraic subgroup of $D_k$
and $\pi_{k,l}(H_k) \subset H_l$ for any $k \geq l \geq 1$.
Then $\varprojlim H_{k}$
is pro-algebraic.
\end{proposition}
Let us remark that a pro-algebraic subgroup $G$ of
$\Diffh\Cd n$ can be expressed in more than one way in the
form $\varprojlim H_{k}$ with the conditions of Proposition~\ref{pro:criterium}. Indeed, if $G=\varprojlim H_{k}$ then $G_k =\overline{\pi_k (G)}$  
is included in $H_k$
but they do not coincide in general.

\subsection{Construction of an infinitesimal generator}
In this section we define infinitesimal generators of formal diffeomorphism and show that for any $F\in\Diffh\Cd n$ there exists an index $m$ such that $F^m$ has an infinitesimal generator (Binyamini \cite{Binyamini:finite}). Before doing so, let us study the relation between the Lie algebra of $\overline{\langle F \rangle}$ and
the group $\overline{\langle F \rangle}_0$.
\begin{lemma}
\label{lem:ccifl}
Let $F \in \Diffh\Cd n$. Then  $\overline{\langle F \rangle}_0 = \Exp({\mathfrak g})$
where ${\mathfrak g}$ is the Lie algebra of $\overline{\langle F \rangle}$.
\end{lemma}
\begin{proof}
Analogously as for linear algebraic groups, since $\langle F \rangle$ is abelian, its Zariski-closure
$\overline{\langle F \rangle}$ is also abelian \cite[Lemma 1]{JR:arxivdl} and then
${\mathfrak g}$ is an abelian Lie algebra \cite[Proposition 3]{JR:arxivdl}.
Since $\overline{\langle F \rangle}_0$  is generated by $\Exp({\mathfrak g})$ by Proposition
\ref{pro:axu}, we have
$\Exp({\mathfrak g}) \subset \overline{\langle F \rangle}_0$. Moreover, any element $L$ of $\overline{\langle F \rangle}_0$ is of the form
\[ L= \Exp(X_1) \circ \cdots \circ \Exp(X_m)  = \Exp (X_1 + \cdots + X_m)  \]
where $X_1, \cdots, X_m \in {\mathfrak g}$. The last equality holds since
${\mathfrak g}$ is abelian.
\end{proof}
\begin{definition}\label{def:infgen}
Given a formal diffeomorphism $F\in\Diffh\Cd n$, a formal vector field $X\in\cvf\Cd n$ is an \emph{infinitesimal generator} of $F$ if $X$ belongs to the Lie algebra of $\overline{\langle F \rangle}$ and $F=\Exp(X)$.
\end{definition}

\begin{theorem}[{\cite[Corollary 7]{Binyamini:finite}}]
\label{theorem:infinitesimal_generator}
Consider $F \in \Diffh\Cd n$. There exists $m \in {\mathbb N}$ such that $F^{m}$ has an infinitesimal generator.
\end{theorem}
\begin{proof}
Since $\overline{\langle F \rangle}_0$ is a finite index normal subgroup of $\overline{\langle F \rangle}$
by Proposition \ref{pro:axu},
there exists $m \in {\mathbb N}$ such that $F^{m} \in \overline{\langle F \rangle}_0$. The result is a
consequence of Lemma \ref{lem:ccifl}.
\end{proof}
The existence of infinitesimal generator is well-known for unipotent formal diffeomorphisms, see for example \cite{Ecalle} or \cite{MaRa:aen}
for the one dimensional case.
\begin{lemma} [{\cite[Theorem 3.17]{Ilya-Y}}] \label{lem:unipotentdiffeo}
Let $F\in \Diffh\Cd n$ be a formal diffeomorphism whose linear part is unipotent, i.e. $\Spec (D_0 F)= \{1\}$. Then 
there exists a unique nilpotent $X \in \cvf\Cd n$ (i.e.
a formal vector field $X$ with $\Spec(D_0 X)= \{0\}$) such that $F = \Exp(X)$. 
\end{lemma}
\begin{remark}\label{rem:unipotentdiffeo}
The formal vector field $X$ in Lemma  \ref{lem:unipotentdiffeo}
belongs to the Lie algebra of $\overline{\langle F \rangle}$ (see \cite[Lemma 1]{JR:arxivdl}), i.e. X is an infinitesimal generator of $F$.
\end{remark}

\subsection{The index of embeddability}\label{sec:index_emb}
Given a formal diffeomorphism $F\in\Diffh\Cd n$, we define the \emph{index of embeddability in a flow} of $F$ as the minimum of the indexes $m\in\N$ such that $F^m$ has an infinitesimal generator. We denote it by $m(F)$, or simply by $m$ when $F$ is implicit. Observe that if $\Spec(D_0F)=\{1\}$ then $m(F)=1$, by Remark~\ref{rem:unipotentdiffeo}. Note also that, by Lemma~\ref{lem:ccifl}, $m(F)$ is the minimum $m$ such that $F^m\in\overline{\langle F \rangle}_0$.
The following remark allows to calculate the index of embeddability of $F$ and its iterates.
\begin{remark}
\label{rem:cc_iter}
Let $F \in  \Diffh\Cd n$ and $r \in {\mathbb Z}^{*}$.
Observe that each coset $F^j \overline{\langle F^r \rangle}$ is a pro-algebraic set, i.e. a projective limit of algebraic sets, since $\overline{\langle F^r \rangle}$ is pro-algebraic.
Moreover, since $\overline{\langle F \rangle}$ is abelian,
we have that the set $\cup_{j=0}^{r-1} F^j \overline{\langle F^r \rangle}$
is a pro-algebraic group and then
$\overline{\langle F \rangle}=\cup_{j=0}^{r-1} F^j \overline{\langle F^r \rangle}$.
Therefore,
$\overline{\langle F^r \rangle}$ is a finite index subgroup of
$\overline{\langle F \rangle}$ and $\overline{\langle F \rangle}/\overline{\langle F^r \rangle}$
is a cyclic group generated by the class of $F$.
Analogously $\overline{\langle F \rangle}/\overline{\langle F^r \rangle}_0$ is cyclic, we just
need to replace  $\cup_{j=0}^{r-1} F^j \overline{\langle F^r \rangle}$
with  $\cup_{j=0}^{s-1} F^j \overline{\langle F^r \rangle}_0$ above where $s$
satisfies $F^s \in \overline{\langle F^r \rangle}_0$.
In particular, we obtain $m(F)=|\overline{\langle F \rangle} : \overline{\langle F \rangle}_0|$.

 It is clear that
$\overline{\langle F^r \rangle}_0 \subset \overline{\langle F \rangle}_0$.
Since  $\overline{\langle F^r \rangle}$ is a finite index subgroup of
 $\overline{\langle F \rangle}$ and $\overline{\langle F^r \rangle}_0$ is a finite index
 subgroup of $\overline{\langle F^r \rangle}$ (Proposition \ref{pro:axu}), we deduce
 that $\overline{\langle F^r \rangle}_0$ is a finite index subgroup of
 $\overline{\langle F \rangle}$. This implies
  $\overline{\langle F \rangle}_0 \subset \overline{\langle F^r \rangle}_0$
  by Proposition \ref{pro:axu} and hence
  $\overline{\langle F \rangle}_0 = \overline{\langle F^r \rangle}_0$.

  By definition $m (F^r)$ is the smallest positive integer $m'$ such that
$F^{r m'} \in  \overline{\langle F^r \rangle}_0 =\overline{\langle F \rangle}_0$.
In particular, we obtain $m(F^r)= m(F) / gcd (m(F),r)$ and $F^r$ has an infinitesimal generator
if and only if  $m(F) $ divides $r$.
\end{remark}

On the other hand, the construction of $\overline{\langle F \rangle}$ implies that the group $D_0 \overline{\langle F \rangle}$ of linear parts
of the elements of $\overline{\langle F \rangle}$ is equal to $\overline{\langle D_0 F \rangle}$, where
the Zariski-closure of the group $\langle D_0 F \rangle$ is considered in $\mathrm{GL}(n,{\mathbb C})$, and we have
\[ |\overline{\langle F \rangle} : \overline{\langle F \rangle}_0| = |\overline{\langle D_0 F \rangle} : \overline{\langle D_0 F \rangle}_0|, \]
(see \cite[Proposition 2.3]{JR:solvable25}). As a consequence, the value of $m(F)=|\overline{\langle D_0 F \rangle} : \overline{\langle D_0 F \rangle}_0|$ depends only on $D_0 F$.
We recall the computation of $|\overline{\langle D_0 F \rangle} : \overline{\langle D_0 F \rangle}_0|$ for the sake of completeness in order to see that it depends only on the
eigenvalues of $D_0 F$.

Given a matrix
$A \in \mathrm{GL}(n,{\mathbb C})$, it admits a unique multiplicative Jordan decomposition
of the form $A = A_s A_u = A_u A_s$ where $A_s, A_u \in \mathrm{GL}(n,{\mathbb C})$, $A_s$
is diagonalizable and $A_u$ is unipotent, i.e. $\mathrm{spec} (A_u)=\{1\}$.
Since $\overline{\langle A \rangle}$ is isomorphic to
$\overline{\langle A _s \rangle} \times \overline{\langle A_u \rangle}$ and
$\overline{\langle A_u \rangle}$ is always connected, we get
\[ \overline{\langle A \rangle} = \overline{\langle A \rangle}_0 \Leftrightarrow
\overline{\langle A_s \rangle} = \overline{\langle A_s \rangle}_0. \]
Up to a linear change of coordinates, we can suppose
$A_s = \mathrm{diag} (\lambda_1, \hdots, \lambda_n)$,
where
$\mathrm{spec}(A)=\{\lambda_1, \hdots, \lambda_n\}$.
Denote $G_{\lambda} = \langle \mathrm{diag} (\lambda_1, \hdots, \lambda_n) \rangle$.
Consider the characters $\chi_{(m_1,\hdots,m_n)}: ({\mathbb C}^{*})^n \to {\mathbb C}^{*}$ of the torus $({\mathbb C}^{*})^n$
given by $\chi_{(m_1,\hdots,m_n)}(\mu_1, \hdots, \mu_n) = \mu_1^{m_1} \hdots \mu_n^{m_n}$ for
$(m_1, \hdots, m_n) \in {\mathbb Z}^{n}$. There exists a one-to-one correspondence between
algebraic subgroups of  $({\mathbb C}^{*})^n$ and subgroups of characters
(see \cite[Theorem 3.2.3.5]{alg.lie:Onis-Vinb}). Indeed given a subgroup $H$ of the group of characters,
the set $G=\cap_{\chi \in H} \mathrm{Ker} (\chi)$ is an algebraic subgroup of
$ ({\mathbb C}^{*})^n$ such that $H= \{ \chi : G \subset  \mathrm{Ker} (\chi) \}$.
As a consequence of the definition of the correspondence, we deduce
\[ \overline{G}_{\lambda} =
\{ \mathrm{diag} (\mu_1, \hdots, \mu_n) : \mu_1^{m_1} \hdots \mu_n^{m_n} = 1 \
\forall (m_1, \hdots, m_n) \in S_{\lambda} \} \]
where 
\begin{equation}\label{eq:setres}
S_{\lambda}= \{(m_1, \hdots, m_n) \in {\mathbb Z}^{n} : \lambda_1^{m_1} \hdots \lambda_n^{m_n} = 1\}
\end{equation}
is the set of resonances of $D_0 F$.  

The Lie algebra
${\mathfrak g}_{\lambda}= \{ \mathrm{diag} (v_1, \hdots, v_n) :  \mathrm{diag} (e^{v_1}, \hdots, e^{v_n})
\in \overline{G}_{\lambda} \}$ of
$\overline{G}_{\lambda}$ is given by
\begin{equation}
\label{equ:form_Lie_alg}
{\mathfrak g}_{\lambda}  =
\{ \mathrm{diag} (v_1, \hdots, v_n) : m_1 v_1 + \hdots + m_n v_n = 0 \ \forall (m_1, \hdots, m_n) \in S_{\lambda}\} .
\end{equation}
Let $S_{\lambda}'$ be the intersection of ${\mathbb Z}^{n}$ with the ${\mathbb Q}$-vector space generated by
$S_{\lambda}$. Notice that we can replace $S_{\lambda}$ with
$S_{\lambda}'$ in equation (\ref{equ:form_Lie_alg}).
Since $\overline{G}_{\lambda,0} =
\mathrm{exp} ({\mathfrak g}_{\lambda})$,
we have
\[ \overline{G}_{\lambda,0} =
\{ \mathrm{diag} (\mu_1, \hdots, \mu_n) : \mu_1^{m_1} \hdots \mu_n^{m_n} = 1 \
\forall (m_1, \hdots, m_n) \in S_{\lambda}' \} . \]
By construction, we get
\[ S_{\lambda}' = \{ (m_1, \hdots, m_n) \in {\mathbb Z}^{n} : \lambda_1^{m_1} \hdots \lambda_n^{m_n} \ \mathrm{is \ a \ root \
 of \ unity} \} . \]
Moreover, $R_{\lambda}=\{ \lambda_1^{m_1} \hdots \lambda_n^{m_n} : (m_1, \hdots, m_n) \in S_{\lambda}' \}$ is a finite subgroup of ${\mathbb S}^{1}$,
$S_{\lambda}$ is a finite index subgroup of $S_{\lambda}'$ and
\[   |\overline{G}_{\lambda}: \overline{G}_{\lambda,0}|  = |S_{\lambda}':S_{\lambda}| = |R_{\lambda}| . \]
Hence, we obtain $m(F) =  |S_{\lambda}':S_{\lambda}| = |R_{\lambda}|$ if
$\mathrm{spec} (D_0 F)= \{\lambda_1, \hdots, \lambda_n \}$.
 \begin{remark}\label{rk:eigenvalues-iteration}
 Consider $F  \in \mathrm{Diff} ({\mathbb C}^{n},0)$ and let $m$ be the index of embeddability of $F$. 
 Since the group $R_{\lambda}$ associated to $ F^{m}$ is the
 trivial group, every eigenvalue of  $D_0 F^{m}$
 that is a root of unity is indeed equal to $1$.
 \end{remark}

\begin{remark}
Assume that $\mathrm{spec} (D_0 F)= \{\lambda_1, \hdots, \lambda_n \}$ and $S_{\lambda}=\{ 0\}$, i.e. the eigenvalues have no resonances. Then $m(F)=1$.
\end{remark}

\subsection{Characterization of the infinitesimal generator}
It is not clear, from Definition~\ref{def:infgen}, whether a formal vector field $X$ such that $F=\Exp(X)$ is an
infinitesimal generator of $F$. Let us show that infinitesimal generators are determined by their linear
parts.
\begin{proposition}
\label{pro:car_inf_gen}
Let $F \in \Diffh\Cd n$. Then
\begin{enumerate}[(i)]
\item $X \in \hat{\mathfrak X} ({\mathbb C}^{n},0)$ is an infinitesimal generator of $F$ if and only if
$F= \Exp(X)$ and $D_0 X$ is in the Lie algebra of the matrix group $\overline{\langle D_0 F \rangle}$.
\item There is a bijective correspondence $X \mapsto D_0 X$ between infinitesimal generators of $F$
and matrices $M$ in the Lie algebra of $\overline{\langle D_0 F \rangle}$ such that
$\Exp(M) = D_0 F$.
\end{enumerate}
\end{proposition}
\begin{proof}
We denote by $G$ the group $\langle F \rangle$ and by ${\mathfrak g}$ the Lie algebra of
$\overline{G}$, which is
abelian by \cite[Proposition 3]{JR:arxivdl}. Let us remark that $G_1$ (see Definition \ref{def:Zariski-closure})  
can be identified with $\overline{\langle D_0 F \rangle}$ and hence
${\mathfrak g}_1$ can be identified with the Lie algebra of  $\overline{\langle D_0 F \rangle}$.

Consider $M \in \mathrm{GL}(n,{\mathbb C})$ such that $M \in {\mathfrak g}_1$ and
$\Exp(M) = D_0 F$. Let us show that there exists $X \in {\mathfrak g}$ such that
$\Exp(X)=F$ and $D_0 X = M$. Since $\pi_{k,l} : \pi_k (G) \to \pi_l (G)$
is surjective, so is $\pi_{k,l}: G_k \to G_l$ and hence
$(d \pi_{k,l})_{\mathrm{Id}}: {\mathfrak g}_k \to {\mathfrak g}_l$ is also surjective for all $k \geq l$.
As a consequence, the map $d \pi_k : {\mathfrak g} \to {\mathfrak g}_k$ is surjective for any
$k \geq 1$. Thus there exists $Y \in {\mathfrak g}$ such that $D_0 Y = M$.
The formal diffeomorphism $F'= F \circ \Exp(-Y)$ is clearly tangent to the identity and
belongs to $\overline{G}$.  There exists a unique  formal vector field $Z$ with vanishing linear part
such that $F' = \Exp(Z)$ by Lemma \ref{lem:unipotentdiffeo}. Notice that $Z$ belongs to
the Lie algebra of $\overline{\langle F' \rangle}$ by Remark \ref{rem:unipotentdiffeo}
and hence to ${\mathfrak g}$ since $\overline{\langle F' \rangle} \subset \overline{\langle F \rangle}$.
Since  ${\mathfrak g}$ is abelian, it follows that
\[ F =\Exp(Z) \circ \Exp(Y) = \Exp(Y+ Z) \]
where $Y+Z \in {\mathfrak g}$ and $D_0 (Y + Z)= D_0 Y = M$.
Therefore the correspondence defined in (ii) is surjective.

We claim that if
$X \in  \hat{\mathfrak X} ({\mathbb C}^{n},0)$, $Y \in  {\mathfrak g}$, $D_0 X = D_0 Y$ and
$\Exp(X)= \Exp(Y) = F$ then $X=Y$. The claim implies that the
correspondence in (ii) is injective.

 Assume the claim is proven. The necessary condition in property (i) is clear.
Let us show the sufficient condition in (i).
Since the correspondence in (ii)
is surjective, there exists $Y \in  {\mathfrak g}$ such that $D_0 Y = D_0 X$ and
$F= \mathrm{exp}(Y)$. Now the claim implies $X=Y$ and hence $X$ belongs to $ {\mathfrak g}$
and is an infinitesimal generator of $F$.

Let us show the claim.
Since $F = \Exp(X)$, we get $F_{*} X = X$.
We define
\[ H_k = \{ L \in D_k : J_k (L_{*} X) = J_k X \}.  \]
It is an algebraic subgroup of $D_k$ since the condition
$J_k (L_{*} X) = J_k X$ can be expressed as a finite number of
algebraic equations in the coefficients of the Taylor series expansion of $L$ of degree less than or equal
to $k$ for any $k \geq 1$.
Clearly, $\pi_{k,l} (H_k) \subset H_l$ is satisfied for all $k \geq l \geq 1$.
We have then that the group
\[ H=  \varprojlim H_k =  \{ L \in \Diffh\Cd n : L_{*} X = X \}  . \]
is a pro-algebraic subgroup of $\Diffh\Cd n$ by Proposition
\ref{pro:criterium}.
Notice that $G$ is a subgroup of $H$ and thus $\overline{G}$ is also a subgroup of  $H$.
Since $Y \in {\mathfrak g}$, we obtain
$\Exp (t Y) \in \overline{G} \subset H$ by Proposition \ref{pro:axu} and hence
$\Exp(t Y)_{*} X = X$ for any $t \in {\mathbb C}$.
This implies $[X, Y]=0$.  We have
\[ \mathrm{Id} = F \circ F^{-1} = \Exp (X) \circ \Exp (-Y) = \Exp(X-Y) \]
where we used $[X,Y]=0$ in the last equality. Since $D_0 (X-Y)=0$, the vector field $X-Y$
is the unique nilpotent vector field whose exponential is the identity map, i.e. $X-Y=0$.
\end{proof}
\begin{definition}
Let $X\in\cvf\Cd n$ with $\mathrm{spec} (D_0 X)= \{ \mu_1, \hdots, \mu_n\}$. We say that
$X$ is \emph{not weakly resonant} if there is no $(m_1, \hdots, m_n) \in {\mathbb Z}^{n}$
such that $\sum_{j=1}^{n} m_j \mu_j \in 2 \pi i {\mathbb Q}^{*}$.
\end{definition}
Similar, but slightly different, conditions of
absence of weak resonances have appeared in the literature
when trying to solve the equation
$\mathrm{exp} (X) =F$ where $F$ and $D_0 X$ are fixed (\cite{Zha}, \cite{rib-embedding}).
Next, we characterize the formal vector fields that are infinitesimal generators.
\begin{theorem}
\label{teo:car_inf_gen}
Let $X\in\cvf\Cd n$ with $\mathrm{spec} (D_0 X)= \{ v_1, \hdots, v_n\}$.
Then $X$ is an infinitesimal generator of $\mathrm{exp}(X)$ if and only if
$X$ is not weakly resonant.
\end{theorem}
\begin{proof}
Let $F= \mathrm{exp}(X)$. The formal vector field $X$ is an infinitesimal generator of $F$ if
and only if  $D_0 X$ is in the Lie algebra of
$\overline{\langle D_0 F \rangle}$ by Proposition \ref{pro:car_inf_gen}.
Let   $D_0 X = S + N$ be the additive Jordan decomposition of $D_0 X$ as a sum of
a semisimple and a nilpotent linear operators that commute. Assume that $S$ is diagonal up
to a change of basis. Notice that $D_0 F = \mathrm{exp}(S) \mathrm{exp}(N)$ is the
multiplicative Jordan decomposition of $D_0 F$.
Denote by ${\mathfrak g}$, ${\mathfrak g}_S$ and ${\mathfrak g}_N$ the Lie algebras of
$\overline{\langle D_0 F \rangle}$,
$\overline{\langle  \mathrm{exp}(S) \rangle}$ and $\overline{\langle  \mathrm{exp}(N) \rangle}$
respectively.
It is well known that  ${\mathfrak g} = {\mathfrak g}_S \oplus {\mathfrak g}_N$
and $ {\mathfrak g}_N$ is the complex vector space generated by $N$
\cite[sections 3.2.2, 3.2.4 and 3.3.7]{alg.lie:Onis-Vinb}.
Moreover, since all the elements of  ${\mathfrak g}_S$ (resp.  ${\mathfrak g}_N$) are
semisimple (resp. nilpotent), ${\mathfrak g}$ is abelian
and the additive Jordan decomposition is unique, we deduce
that ${\mathfrak g}_S$ (resp.  ${\mathfrak g}_N$) is the set of semisimple (resp. nilpotent)
elements of ${\mathfrak g}$.
As a consequence, $D_0 X$ is in the Lie algebra of $\overline{\langle D_0 F \rangle}$
if and only if $S$ is in the Lie algebra of $\overline{\langle  \mathrm{exp}(S) \rangle}$.
Notice that  
\[ S = \diag (v_1, \hdots, v_n) \ \ \mathrm{and} \ \
\mathrm{exp} (S) = \diag (\lambda_1, \hdots, \lambda_n), \]
where $\lambda_j = e^{v_j}$ for $1 \leq j \leq n$.
Since  $\lambda_1^{m_1} \hdots \lambda_n^{m_n} = e^{m_1 v_1 + \hdots + m_n v_n}$, we deduce
\[ S_{\lambda} = \{
(m_1, \hdots, m_n) \in {\mathbb Z}^n :
m_1 v_1 + \hdots + m_n v_n \in 2 \pi i {\mathbb Z}
\} \]
where $S_{\lambda}$ is defined in equation~\eqref{eq:setres}.
Recall that ${\mathfrak g}_S=\mathfrak g_{\lambda}$ is given by equation
(\ref{equ:form_Lie_alg}).
It is clear that if $X$ is not weakly resonant then $S \in {\mathfrak g}_S$.

Suppose $S \in {\mathfrak g}_S$. Let $(m_1, \hdots, m_n)$ with
$\sum_{j=1}^{n} m_j v_j \in 2 \pi i {\mathbb Q}$. Up to multiplication by a non-zero integer, we
can assume $\sum_{j=1}^{n} m_j v_j \in 2 \pi i {\mathbb Z}$.
This implies $\lambda_1^{m_1} \hdots \lambda_n^{m_n} =1$ and hence
$(m_1, \hdots, m_n) \in S_{\lambda}$.
We deduce that $\sum_{j=1}^{n} m_j v_j =0$ by the description of ${\mathfrak g}_S$.
\end{proof}

In the following proposition we prove that the infinitesimal generator is unique only in the unipotent case. This is the reason justifying why infinitesimal generators have been considered exclusively for unipotent diffeomorphisms in the literature.
\begin{proposition}
Consider $F\in\Diffh\Cd n$ such that $m(F)=1$. The infinitesimal generator of $F$ is unique if and only if $F$ is unipotent. Otherwise, $F$ has infinitely many infinitesimal generators.
\end{proposition}
\begin{proof}
The uniqueness of the infinitesimal generator is equivalent to the uniqueness of the infinitesimal
generator of
$\mathrm{diag} (\lambda_1, \hdots, \lambda_n)$ in the group of diagonal matrices where
$\mathrm{spec} (D_0 F) = \{\lambda_1, \hdots, \lambda_n\}$.
Let us recall that the conditions $m(F)=1$,
$\overline{\langle F \rangle} = \overline{\langle F \rangle}_0$ and
$S_\lambda = S_{\lambda}'$ are equivalent (see Section~\ref{sec:index_emb}).
Given an infinitesimal generator $\mathrm{diag} (\mu_1, \hdots, \mu_n)$ of
$\mathrm{diag} (\lambda_1, \hdots, \lambda_n)$, the infinitesimal generators of
$\mathrm{diag} (\lambda_1, \hdots, \lambda_n)$ are the matrices of the form
$\mathrm{diag} (\mu_1 + 2 \pi i k_1, \hdots, \mu_n + 2 \pi i k_n)$ where
$k_1, \hdots, k_n$ are integer numbers such that
$m_1 k_1 + \dots + m_n k_n =0$ for any
$ (m_1, \hdots, m_n) \in S_{\lambda}'$. The system of equations has either a unique solution,
if the rank of $ S_{\lambda}'$ is equal to $n$, or infinitely many solutions otherwise.
Notice that the rank of $S_{\lambda}'$ is equal to $n$ if and only if
$ S_{\lambda}'={\mathbb Z}^{n}$ and this condition is equivalent to
$\lambda_1=\dots=\lambda_n=1$ since $S_{\lambda}=S_{\lambda}'$.
As a consequence $F$ has a unique infinitesimal
generator if $F$ is unipotent and infinitely many infinitesimal generators otherwise.
\end{proof}
\subsection{Geometrical properties of the infinitesimal generator}

In this section, we prove that if a diffeomorphism $F$ has an infinitesimal generator $X$, then $F$ and $X$ have the same formal analytic invariant sets (ideals). Recall that an ideal $I\subset \hat{\mathcal{O}}_n$ is invariant for a diffeomorphism $F\in\Diff \Cd n$ if $I\circ F\subset I$, and is invariant for a formal vector field $X\in\cvf\Cd n$ if $X(I)\subset I$.

\begin{proposition}
\label{pro:group_pro}
Given an ideal $I\subset\hat{\mathcal{O}}_n$, the set
\[ {\mathcal I}_{I} = \{ F \in \Diffh\Cd n : I\circ F \subset I \}  \]
is a pro-algebraic group.  Moreover, it satisfies
${\mathcal I}_{I} = \{ F \in \Diffh\Cd n : I\circ F = I \}$.
\end{proposition}
\begin{proof}
In order to show that  ${\mathcal I}_{I}$ is pro-algebraic, we will use Proposition
\ref{pro:criterium}. The ideal $I$  is of the form  $I=(f_1, \hdots, f_m)$ since
${\mathbb C}[[x_1,\hdots,x_n]]$ is noetherian.  We define
\[ H_k = \{ F \in D_k : f_j \circ F \in I + {\mathfrak m}^{k+1} \ \forall 1 \leq j \leq m \}.  \]
It is clear that $\pi_{k,l}(H_k) \subset H_l$ for any $k \geq l \geq 1$.
The inclusion ${\mathcal I}_{I} \subset  \varprojlim H_{k}$ is obvious.
Any element $F$ of $D_k$ can be interpreted as the $k$-th jet of a local
diffeomorphism. Hence,
the power series $f_j \circ F$ is of the form
$\sum_{i_1, \hdots, i_n} a_{i_1, \hdots, i_n}^{j} x_1^{i_1} \hdots x_n^{i_n}$
where every $a_{i_1, \hdots, i_n}^{j}$ is a polynomial  in the coefficients of the Taylor
series expansion of $F$ at the origin. The condition
$f_j \circ F \in I + {\mathfrak m}^{k+1}$ is satisfied if and only if
$\sum_{i_1+ \hdots + i_n \leq k } a_{i_1, \hdots, i_n}^{j} x_1^{i_1} \hdots x_n^{i_n}$
belongs to the complex vector space $V_k$ generated by the polynomials of the form
\[ J_k (x_1^{i_1} \dots x_n^{i_n} f_j) \  \mathrm{where} \ i_1+ \dots + i_n \leq k \ \mathrm{and} \
1 \leq j \leq m .\]
The property $J_k (f_j \circ F) \in V_k$ is equivalent to a linear system of equations on the
coefficients  $a_{i_1, \hdots, i_n}^{j}$ with $i_1+ \dots + i_n \leq k$ by elementary linear algebra.
Therefore the condition  $J_k (f_j \circ F) \in V_k$ is equivalent to a system of polynomial
equations in the coefficients of $F$.
Hence $H_k$ is an algebraic subset of $D_k$.

We denote by $I_k$ the natural projection of $I$ in
${\mathfrak m}/{\mathfrak m}^{k+1}$. The definition of $H_k$ implies
\[ H_k = \{ F \in D_k : F( I_k) \subset I_k \}, \]
(see equation (\ref{equ:comp})).
Since $F$ defines an element of  $\mathrm{GL}({\mathfrak m}/{\mathfrak m}^{k+1})$ and
${\mathfrak m}/{\mathfrak m}^{k+1}$ is finite dimensional, it follows that
\[ H_k = \{ F \in D_k :F( I_k) = I_k \} . \]
As a consequence $H_k$ is a subgroup of $D_k$ and hence $H_k$ is an algebraic subgroup of
$D_k$ for any $k \geq 1$.

Finally, let us show ${\mathcal I}_{I} = \varprojlim H_{k}$.
By definition, we have
\[ \varprojlim H_{k} = \{ F \in \Diffh\Cd n : I \circ F + {\mathfrak m}^{k+1} =
I + {\mathfrak m}^{k+1} \ \forall k \geq 1 \} . \]
Since $\cap_{k=1}^{\infty} (J + {\mathfrak m}^{k+1}) = J$ for any proper ideal $J$ of
${\mathbb C}[[x_1,\hdots,x_n]]$ by Krull's intersection theorem, we deduce
$ \varprojlim H_{k} =  \{ F \in \Diffh\Cd n : I \circ F  = I  \}$ and hence
$ \varprojlim H_{k}  \subset {\mathcal I}_{I}$.
Since $ {\mathcal I}_{I} \subset \varprojlim H_{k} $, we obtain $ \varprojlim H_{k}  = {\mathcal I}_{I}$.
Therefore ${\mathcal I}_{I}$
is a pro-algebraic group by Proposition \ref{pro:criterium}.
\end{proof}
\begin{proposition}
Given a formal curve $\G$, the group
$${\mathcal I}_{\Gamma}'=\{F\in \mathcal{I}_{\Gamma}: (F|_\G)'(0)=1\}$$
is pro-algebraic.
\end{proposition}
\begin{proof}
Let $\gamma(t) = (\gamma_1(t), \hdots, \gamma_n(t))$ be an irreducible parametrization of $\Gamma$.
We denote $J_{\nu} \gamma(t) = (\mu_1 t^{\nu}, \hdots, \mu_n t^{\nu})$ where
$\nu$ is the multiplicity of $\G$.
We define the auxiliary group
\[ {\mathcal J} = \{ F \in \Diffh\Cd n:
(\mu_1, \hdots, \mu_n) \in \ker (D_0 F - \mathrm{Id}) \}  \]
that is clearly pro-algebraic.
It is straightforward to check out the inclusion
${\mathcal I}_{\Gamma}' \subset {\mathcal J}$.
Since the intersection of pro-algebraic groups is pro-algebraic
\cite[Remark 2.7]{JR:solvable25}, we obtain that
$\hat{\mathcal I}_{\Gamma} = {\mathcal I}_{\Gamma} \cap {\mathcal J}$ is a pro-algebraic
group containing ${\mathcal I}_{\Gamma}'$.
In order to show that ${\mathcal I}_{\Gamma}' $ is pro-algebraic, consider the morphism of groups
\[ \begin{array}{ccccc}
A & : &  \hat{\mathcal I}_{\Gamma} & \to & {\mathbb C}^{*} \\
& & F & \mapsto &  (F|_{\Gamma})' (0) .
\end{array} \]
Notice that the tangent value $\lambda(\Gamma)$ is equal to 1 for any element of $\hat{\mathcal I}_\Gamma$. Using the relation between the
tangent and the inner eigenvalues given in equation
(\ref{equ:inner_tangent}), we deduce that
the image of $A$ is contained in the group of roots of  unity of order $\nu$.
Therefore  ${\mathcal I}_{\Gamma}' $ is a finite index subgroup of
$\hat{\mathcal I}_{\Gamma}$ and hence pro-algebraic by
Proposition \ref{pro:axu}.
\end{proof}

The next results are consequences  
of the above proposition and the general properties of
pro-algebraic groups.
\begin{proposition}
\label{pro:remres}
Let $F \in \Diffh\Cd n$ and let $I\subset\hat{\mathcal{O}}_n$ be an ideal. Suppose that there exists
$r \in {\mathbb Z}^{*}$ such that $I$ is invariant for $F^{r}$.  Then
$\overline{\langle F \rangle}_0 \subset {\mathcal I}_{I}$. Moreover, if $I=\G$ is a formal curve and $(F^{r}|_{\Gamma})' (0)$
is a root of unity then $\overline{\langle F \rangle}_0 \subset {\mathcal I}_{\Gamma}'$.
\end{proposition}
\begin{proof}
Since ${\mathcal I}_{I}$ is pro-algebraic, we obtain
$\overline{\langle F^{r} \rangle} \subset {\mathcal I}_{I}$. Since $\overline{\langle F \rangle}_0 =
\overline{\langle F^r \rangle}_0$ by Remark \ref{rem:cc_iter}, we get
$\overline{\langle F \rangle}_0 \subset {\mathcal I}_{I}$.
If $I=\G$ is a formal curve and $(F^{r}|_{\Gamma})' (0)$  is a root of unity then we can replace $r$ with a multiple to obtain
$(F^{r}|_{\Gamma})' (0) =1$. The same proof shows
$\overline{\langle F \rangle}_0 \subset {\mathcal I}_{\Gamma}'$.
\end{proof}
\begin{proposition}
\label{pro:perbinv}
Let $F \in \Diffh\Cd n$ and let $m$ be the index of embeddability of $F$. Let $X$ be an infinitesimal generator of $F^{m}$. Given an ideal $I$ of $\hat{\mathcal{O}}_n$, the following properties are equivalent:
\begin{enumerate}
	\item $I$ is invariant for $X$;
	\item $I$ is invariant for $F^m$;
	\item $I$ is invariant for a non-trivial iterate of $F$.
\end{enumerate}
\end{proposition}
\begin{proof}
The implications $(1)\Rightarrow(2)$ and $(2)\Rightarrow(3)$ are clear. Assume that (3) holds. We have  $\Exp(t X) \in  \overline{\langle F \rangle}_0$ for any $t \in {\mathbb C}$
by Proposition \ref{pro:axu}.
Since
$\overline{\langle F \rangle}_0 \subset {\mathcal I}_{I}$ by Proposition \ref{pro:remres},
we obtain $\Exp(t X) \in {\mathcal I}_{I}$ for any $t \in {\mathbb C}$.
Thus $I$ is invariant for $X$.
\end{proof}

As a consequence of Proposition~\ref{pro:perbinv}, we recover the following result of Rib\'{o}n:
\begin{cor}[\cite{Rib}]
\label{cor:per_curve}
Let $F \in \diff{}{2}$ and let $m=m(F)$ be its index of embeddability. Then there exists a formal $m$-periodic curve of $F$.
\end{cor}
\begin{proof}
The diffeomorphism $F^{m}$ has an infinitesimal generator $X$. The formal version of Camacho-Sad's
theorem \cite{Cam-S} provides a formal invariant curve $\Gamma$ that is invariant by
$X$. Thus $\Gamma$ is invariant by $F^{m}$.
\end{proof}

\begin{proposition}
\label{pro:remres2}
Let $F \in \Diffh\Cd n$ and let $\Gamma$ be a formal curve. Suppose that there exists
$r \in {\mathbb Z}^{*}$ such that $\Gamma$ is invariant for $F^{r}$ and $(F^{r}|_{\Gamma})' (0)$ is a root of
unity. Let $X$ be an infinitesimal generator of $F^{m}$, where $m$ is the index of embeddability of $F$. Then the inner eigenvalue of $(X,\G)$ is equal to $0$. In particular
the tangent line of $\Gamma$ is contained in the kernel of the linear part $D_0X$ of $X$ at the origin.
\end{proposition}
\begin{proof}
Proposition \ref{pro:remres} implies $\overline{\langle F \rangle}_0 \subset {\mathcal I}_{\Gamma}'$
and hence $\Exp(t X) \in  {\mathcal I}_{\Gamma}'$ for any $t \in {\mathbb C}$, that is, the inner eigenvalue of the pair $(\Exp(tX),{\Gamma})$ is equal to $1$ for any
$t$.
 Let $\gamma (s)$ be an irreducible parametrization of $\G$.
We have that $\Exp(tX) (\gamma (s))$ is of the form $\g \circ \phi_{t} (s)$
where $\phi_{t}(0)=0$ and $\phi_{t}' (0)=1$ for any $t \in {\mathbb C}$.
Since $X(\g (s)) = \frac{\partial \g \circ \phi_{t}(s)}{\partial t}\big|_{t=0}$, it follows that
the multiplicity of the right hand side is at least $\nu (\g) +1$ and thus
$\nu (X|_{\Gamma}) = \nu (X(\gamma(s))) - \nu (\g' (s))  \geq 2$.
Since $\nu (X|_{\Gamma}) >  \nu (\g (s))$, any
non-zero tangent vector $v$  of $\Gamma$ at $0$ is in the kernel of $D_0X$.
\end{proof}

\subsection{Examples of diffeomorphisms possessing a formal invariant curve}\label{sec:cond_for_inv_curves}
Corollary \ref{cor:per_curve} provides a formal periodic curve for any $F \in \diff{}{2}$.
This is no longer true for dimension greater than $2$.
More precisely, there exist nilpotent analytic vector fields $X \in {\mathfrak X} \cn{3}$ with no
formal invariant curve by a theorem of G\'{o}mez Mont and Luengo  \cite{Gom-L}.
Then the diffeomorphism $\Exp(X)$ has no formal periodic curve by Proposition \ref{pro:perbinv}.

In this section we apply our results about infinitesimal generators to obtain conditions that guarantee the
existence of a formal periodic curve in dimension $n=3$.

A formal codimension 1 foliation $\mathcal{F}_\omega$ in $(\C^3,0)$ is determined by a non-zero 1-form
$$\omega = a_1 dx_1 + a_2 dx_2 + a_3 dx_3,\;\;
a_1, a_2, a_3 \in \hat{\mathcal{O}}_3,$$
 satisfying the integrability condition $\omega \wedge d \omega =0$.
Two 1-forms $\omega$ and $\omega'$ define the same foliation if there exists
$f \in \hat{K}_3 \setminus \{0\}$ such that $\omega = f \omega'$ where $\hat{K}_3$ is the field
of fractions of $\hat{\mathcal{O}}_3$.
We say that $\mathcal{F}_\omega$ has a formal integrating factor if there exists
$f \in \hat{K}_3 \setminus \{0\}$ such that
$d \left( \frac{\omega}{f} \right) =0$.
\begin{proposition}
\label{pro:curve3}
Let $F \in \diff{}{3}$. Suppose that either
\begin{enumerate}
	\item there exists a foliation ${\mathcal F}_{\omega}$
with no formal integrating factor such that
 $F^{*}\omega \wedge \omega =0$ or
	\item there exists $g \in \hat{\mathcal{O}}_3 \setminus {\mathbb C}$ such that $g \circ F = g$.
\end{enumerate}
Then $F^{m}$ has a formal invariant curve,
where $m$ is the index of embeddability of $F$.
\end{proposition}
The two cases are of different nature. Namely, in case (1) we are requiring that $F$ preserves a
foliation with ``poor" integrability properties whereas in  case (2) we are asking $F$ to preserve
the ``fibers" of $g$.
\begin{lemma}
\label{lem:inf_tangent}
Suppose that the hypotheses of Proposition \ref{pro:curve3} are satisfied.
Set  $\omega= dg$ in case (2).
Then $\omega (X)=0$
for any infinitesimal generator $X$ of $F^{m}$.
\end{lemma}
\begin{proof}
Assume that we are in case (1).  Analogously as in Proposition \ref{pro:group_pro},
we can show that the group
\[ {\mathcal I}_{\omega} = \{ L  \in \diffh{}{3} : L^{*} \omega \wedge \omega =0 \} \]
is of the form $\varprojlim H_{k}$,
where
\[ H_k = \{ L \in D_k : J_k (L^{*} \omega) = J_k (h_k \omega) \ \mathrm{for \ some} \
h_k \in \hat{\mathcal{O}}_3 \} \]
is an algebraic subgroup of $D_k$ and $\pi_{k,l} (H_k) \subset H_l$ for all $k \geq l \geq 1$.
Therefore ${\mathcal I}_{\omega}$ is pro-algebraic by Proposition \ref{pro:criterium}.
We deduce that $\overline{\langle F \rangle}$ is contained in ${\mathcal I}_{\omega}$ and
hence $\Exp(t X)^{*} \omega \wedge \omega =0$ for any infinitesimal generator $X$ of $F^m$ and any $t \in {\mathbb C}$ by
Proposition \ref{pro:axu}. As a consequence the Lie derivative $L_{X} \omega$ satisfies
$L_{X} \omega \wedge \omega =0$. This implies either that $\omega (X)=0$ or that $\omega (X)$
is a formal integrating factor of $\omega$ (see \cite[Chapitre III, Proposition 1.3]{Cer-M}). Since the latter possibility is excluded
by hypothesis, we obtain  $\omega (X)=0$.

Assume that we are in case (2). The group
\[ {\mathcal I}_{g} = \{ L  \in \diffh{}{3} : g \circ L = g \} \]
is of the form $\varprojlim H_{k}$,
where $H_k = \{ L \in D_k : J_k  (g \circ L)  = J_k g \}$
is an algebraic subgroup of $D_k$ and $\pi_{k,l} (H_k) \subset H_l$ for all $k \geq l \geq 1$.
Arguing as in the previous case, we obtain $g \circ \Exp(t X)=g$ for any infinitesimal generator $X$ of $F^m$ and any $t \in {\mathbb C}$. We get
$\omega (X)= dg (X) = X(g)=0$.
\end{proof}
Proposition \ref{pro:curve3} is an immediate consequence of Lemma \ref{lem:inf_tangent}
and next result.
\begin{proposition}
Let $X \in\cvf\Cd3$. Consider a formal codimension $1$
foliation $\mathcal{F}_\omega$ such that
$\omega (X)=0$. Then $X$ has a formal invariant curve.
\end{proposition}
This result is due to Cerveau and Lins Neto \cite[Proposition 3]{Cerveau-Lins:cod2} for holomorphic
foliations and vector fields. We just adapt their proof to the formal setting.
\begin{proof}
	If $\alpha$ is a formal differential form or a formal vector field, we denote by $\Sing(\alpha)$ the ideal of its coefficients in $\hat{\mathcal{O}}_3$.
We can suppose that the coefficients of $\omega$ have no common factor up to divide $\omega$
by the $\gcd$ of such coefficients. In other words we have
$\codim (\mathrm{Sing} (\omega)) \geq 2$.
We can assume $\dim(\mathrm{Sing} (X))=0$, otherwise the result is trivial. Moreover, this
implies $\omega (0) =0$ since otherwise the foliation $\mathcal{F}_\omega$ is equal to
$\mathcal{F}_{dx}$ up to a formal change of coordinates
and $dx (X)=0$ implies $X = b(x,y,z) \partial_y + c(x,y,z) \partial_z$
and then $\dim (\mathrm{Sing} (X)) \geq 1$.

Denote $\eta = i_{X} (dx \wedge dy \wedge dz)$.
The property $\omega (X)=0$ is equivalent to $\omega \wedge \eta =0$.
We claim  $\codim (\mathrm{Sing}(\omega))  \neq 3$, otherwise we can apply the
de Rham-Saito lemma \cite{Saito:lemma} to show that the $2$-form $\eta$ is of the form
$\omega \wedge \theta$ (the result works both in the formal and analytic settings).
We have
\[ \mathrm{Sing} (\omega \wedge \theta) =  \mathrm{Sing} (\eta) =  \mathrm{Sing}  (X) . \]
Hence $\dim(\Sing(\omega\wedge\theta))=0$. There exists $k \in {\mathbb N}$ such that if $\omega', \theta'$ are formal $1$-forms
such that $J_k \omega = J_k \omega'$ and $J_k \theta = J_k \theta'$ then
$\dim(\mathrm{Sing} (\omega' \wedge \theta'))=0$.
In particular, we get $\dim(\mathrm{Sing}  (J_k \omega \wedge J_k \theta))=0$.
This provides a contradiction since it is known that the codimension of the singular set of
the exterior product of two germs of holomorphic $1$-form has codimension less than or equal
to $2$ if it is singular at $(0,0,0)$ (see \cite[Lemma 3.1.2]{Medeiros:pforms}).

We deduce $\codim (\mathrm{Sing} (\omega)) = 2$ and hence
$\mathrm{Sing} (\omega)$ is a formal curve.
Let us remark that since $\omega(X)=0$, 
the curve  $\mathrm{Sing} (\omega)$ is invariant by $X$.
\end{proof}

\section{Reduction to Ramis-Sibuya form}\label{sec:RS-vector-fields}

In this section, we show that a pair $(F,\Gamma)$, where $F$ is a diffeomorphism and $\Gamma$ is a rationally neutral formal invariant curve of $F$ not contained in the set of fixed points of a non-trivial iterate of $F$, can be reduced, up to iterating $F$, to a pair $(\tilde{F},\tilde{\Gamma})$ in Ramis-Sibuya form. First, we perform such a reduction in the context of formal vector fields. Next, we use the results in Section~\ref{sec:infinitesimal} to adapt the reduction to diffeomorphisms.

\subsection{Ramis-Sibuya form for formal vector fields}\label{sec:RS-vector-fields-definition}
\begin{definition}\label{def:X-RS-form}
Let $X$ be a singular formal vector field at $0\in\C^n$ and let $\G$ be a formal invariant curve of $X$. We say that the pair $(X,\G)$ is in {\em Ramis-Sibuya form} (\emph{RS-form} for short) if $\G$ is non-singular and there exist analytic coordinates
$(x,\yy)$ at $0\in\C^n$ for which $\G$ is transversal to the hyperplane $x=0$ and such that $X$ is written as
\begin{equation}\label{eq:X-RS-form}
X=x^{q+1}(\lambda+b x^{\max (1,q)} + x^{q+1}
A(x,\yy))\frac{\partial}{\partial
x}+\left((D(x)+x^q C)\yy+x^{q+1}B(x,\yy)\right)\frac{\partial}{\partial\yy},
\end{equation}
where $q\geq 0$, $\lambda\in\C^*$, $b \in {\mathbb C}$,
$A(x,\yy)\in\C[[x,\yy]]$, $B(x,\yy)\in\C[[x,\yy]]^{n-1}$ and
\begin{enumerate}[(i)]
\item $D(x)$ is a diagonal  matrix of polynomials of
degree at most $q-1$ (equal to $0$ if $q=0$) and $C$ is a constant matrix,
\item $D(x)+x^{q}C\not\equiv 0$,
\item $D(x)$ commutes
 with $C$.
 \end{enumerate}
 The polynomial vector field $\lambda x^{q+1}\frac{\partial}{\partial x}+(D(x)+x^qC)\yy\frac{\partial}{\partial\yy}$ is called the {\em principal part} of $(X,\G)$ in the coordinates $(x,\yy)$.
 \end{definition}

Notice that $q+1$ is the multiplicity of the restricted vector field $X|_\G$ and thus the integer $q=q(X,\G)$ is well defined for the pair $(X,\G)$ and is independent of the coordinates. On the other hand, if the multiplicity of $X$ is $\nu(X)=\nu+1$ then $\nu\le q$ and $\nu$ coincides with the order at $x=0$ of the polynomial matrix $D(x)+x^qC$. Thus, the number $p=p(X,\G)=q-\nu\ge 0$, called the {\em Poincar\'{e} rank} of the pair $(X,\G)$, is also independent of the coordinates.

\begin{remark}\label{rk:RS-form} Assume that $(X,\G)$ is in RS-form, written
as (\ref{eq:X-RS-form}) in coordinates $(x,\yy)$.
\begin{enumerate}[(a)]
\item $\G$ is not contained in the singular locus of the vector field $X$.
\item Let $l\ge 1$ be the order of contact of $\G$ with the
$x$-axis, i.e. $\G$ admits a parametrization $(s,\bar{\g}(s))\in\C[[s]]^{n}$ where the minimum order of the components of $\bar{\g}(s)$ is equal to $l$. Then the invariance condition implies that the order in $x$ of any component of the vector $X(\yy)(x,0)\in\C[[x]]^{n-1}$
is at least $l+\nu$.
\item If $q\ge 1$ then, after a change of variables of the form $\bar{x}=ax$ where $a^q=-\lambda$, we may assume that $\lambda=-1$.
\item Denote by $Q_1(x),...,Q_l(x)$ the different polynomials in the diagonal of the matrix $D(x)$ and, up to reordering the $\yy$-variables, write $$D(x)=\diag(Q_1(x)I_{n_1},\dots,Q_l(x)I_{n_l}).$$
 The property $[D(x),C]=0$ implies that $C$ is block-diagonal $C=\diag(C^{1},\dots,C^{l})$ where $C^{j}$ has size $n_j$. After a linear change of variables of the form $\bar{\yy}=P\yy$, we may assume that the blocks of the matrix $C$ are in Jordan canonical form.
\end{enumerate}
\end{remark}

Let us justify our choice of the terminology in Definition~\ref{def:X-RS-form}. After dividing the
vector field in (\ref{eq:X-RS-form}) by $x^\nu$ times a unit, we can associate it to a system of $n-1$ formal ODEs
$$x^{p+1}\yy'=(\overline{D}(x)+x^p\overline{C}+O(x^{p+1}))\yy+O(x^{p+1}),$$
where $\overline{D}(x)=D(x)/\lambda x^\nu$ and $\overline{C}$ is a constant matrix. Such a system has a singular point at $x=0$ with {\em Poincar\'{e} rank}
equal to $p$ (unless possibly for $q=p=0$ if $\overline{C}=0$). Moreover, the properties assumed for the polynomial matrix
$D(x)+x^pC$ are essentially those considered in the work
of Ramis and Sibuya \cite{Ram-S}, devoted to proving {\em multisummability} of the
formal solution $\yy=\bar{\g}(x)$, where $(x,\bar{\g}(x))$ is a parametrization of $\G$, in the case where the coefficients of
the system are convergent.

\subsection{Blow-ups and ramifications along an invariant curve}

Let $X\in\cvf\Cd n$ be a singular vector field and let $\G$ be a formal invariant curve of $X$ not
contained in the singular locus of $X$.

 A germ of holomorphic map
$\phi:(\C^n,0)\to(\C^n,0)$ will be called a {\em permissible
transformation} for the pair $(X,\G)$ if it is of one of the
following types:

\begin{enumerate}
\item[1.] The germ of a holomorphic diffeomorphism.
\item[2.] Let $Z$ be a germ of non-singular analytic submanifold at $0\in\C^n$ which is invariant for $X$ (meaning that $X(g)\in I(Z)$ for any $g\in I(Z)$,
where $I(Z)$ denotes the ideal of holomorphic germs vanishing on
$Z$) and
 such that the tangent line of $\Gamma$ is transversal to $Z$.
  Let $\pi_Z:M\to U$ be the blow-up with center $Z$ and
  let $p\in\pi_Z^{-1}(0)$ be the point corresponding to the tangent of
  $\G$. Then there is
  an analytic chart $\tau$ of $M$ at $p$
  so that $\phi$ is the germ of $\pi_Z\tau^{-1}$ at $0\in\C^n$. We will
  say that $Z$ is a {\em permissible center} and that $\phi$ is a {\em permissible blow-up}.
    \item[3.] The curve $\G$ is non-singular, there are
    analytic coordinates $\zz=(z_1,...,z_n)$ at $0\in\C^n$
     such that $Z=\{z_1=0\}$ is invariant for $X$ and transversal to $\G$ and $\phi$ is the
    map
    $
    \phi(\zz)=(z_1^l,z_2,...,z_n)
    $ for some $l\in\N_{>0}$. We will
  say that $\phi$ is a {\em permissible $l$-ramification} (with {\em center} $Z$).
\end{enumerate}
In the last two cases, the non-singular hypersurface
$E_\phi=\phi^{-1}(Z)$ is called the {\em exceptional divisor of
$\phi$}. For convenience, $E_\phi=\{0\}$ in the case where
$\phi$ is a diffeomorphism. Notice that a permissible
transformation $\phi$ is a local diffeomorphism at every point in
the complement of $E_\phi$.

The following result is quite well known (see for instance \cite{Can-M-R} for the three-dimensional case). We include a proof for the sake of completeness.

\begin{proposition}\label{pro:permissible-X}
Let $\phi:\Cd n\to\Cd n$ be a permissible transformation for $(X,\G)$. There exist a unique formal curve $\wt{\G}$ at
$0\in\C^n$ such that $\phi^*\G\subset\wt{\G}$ (where
$\phi^*\G=\{g\circ\phi\,:\,g\in\G\}$) and a unique formal vector field $\wt X$ at $0\in\C^n$ such that $\phi_*\wt X=X$. Moreover, $\wt{X}$ is singular and has $\wt\G$ as an invariant curve. In addition, the multiplicities of the restrictions satisfy $\nu(\wt X|_{\wt\G})\ge\nu(X|_\G)$. We will call $\wt{X}$ and $\wt\G$ the {\em transforms} of $X$ and $\G$ by $\phi$, respectively.
\end{proposition}

\begin{proof}
The case where $\phi$ is a germ of a diffeomorphism is clear.

Suppose that $\phi$ is a permissible blow-up with center $Z$. Consider
analytic coordinates $\zz=(z_1,z_2,...,z_n)$ so that the tangent line of
$\G$ is tangent to the $z_1$-axis and
$Z=\{z_1=z_2=\cdots=z_t=0\}$ where $t=\codim Z$ (thus $I(Z)$ is
generated by $z_1,...,z_t$). We may write $\phi:\Cd n\to\Cd n$ as
\begin{equation}\label{eq:expression-blow-up}
\phi(\zz)=(z_1,z_1z_2,...,z_1z_t,z_{t+1},...,z_n).
\end{equation}
Let $\g(s)=(\g_1(s),...,\g_n(s))\in\C[[s]]^n$ be an irreducible parametrization
of $\G$ in the coordinates $\zz$. Then
$
\nu(\G)=\nu(\g_1(s))<\nu(\g_j(s))$ for $j=2,...,n,$
where $\nu$ denotes the order in $s$. Also,
\begin{equation}\label{eq:gamma-tilde}
\wt{\g}(s)=
\Bigl(\g_1(s),\frac{\g_2(s)}{\g_1(s)},...,\frac{\g_t(s)}{\g_1(s)},\g_{t+1}(s),...,\g_n(s)\Bigr)
\in\C[[s]]^n
\end{equation}
is a parametrization of a formal curve $\wt{\G}$ which satisfies
$\phi^*\G\subset\wt{\G}$. The uniqueness of $\wt{\G}$ can be seen
as follows: if
$\bar{\g}(s)=(\bar{\g}_1(s),\bar{\g}_2(s),...,\bar{\g}_n(s))$ is a
parametrization of another formal curve $\bar{\G}$ satisfying
$\phi^*\G\subset\bar{\G}$ then we will have that
$\phi\circ\bar{\g}(s)$ is another parametrization of $\G$ and
necessarily $\phi\circ\bar{\g}(s)=\g(\sigma(s))$ where
$\sigma(s)\in\C[[s]]$. Using the expression of $\phi$ and equation
(\ref{eq:gamma-tilde}) one shows that
$\bar{\g}(s)=\wt{\g}(\sigma(s))$ and thus $\overline{\G}=\wt{\G}$.

Write $X=\sum_{i=1}^n a_i(\zz)\frac{\partial}{\partial z_i}$. Since $\G$ is invariant and not contained in the singular locus of
$X$, we have that the vector $X|_{\g(s)}\in\C[[s]]^{n}$ is a non-zero multiple of $\g'(s)$ and hence $\nu(a_1(\g(s)))<\nu(a_j(\g(s)))$ for $j=2,...,n$. So $a_j(\zz)$ cannot contain a monomial of the form $cz_1$, with $c\neq0$, for $j=2,...,n$.
On the other hand, the condition of $Z$ being invariant implies
that, for $j=1,...,t$, $a_j(\zz)\in I(Z)$ and hence
$a_j(\phi(\zz))$ is divisible by $z_1$. Using these two properties, the
vector field $\wt{X}=\sum_{i=1}^n
\tilde{a}_i(\zz)\frac{\partial}{\partial z_i}$ defined by
\begin{equation}\label{eq:X-after-blowup}
\left\{
  \begin{array}{ll}
    \tilde{a}_j(\zz)=\dfrac{a_j(\phi(\zz))-z_ja_1(\phi(\zz))}{z_1}, &
 \hbox{for $j=2,...,t$;} \\[7pt]
    \tilde{a}_j(\zz)=a_j(\phi(\zz)), & \hbox{for $j\in\{1,t+1,...,n\}$,}
  \end{array}
\right.
\end{equation}
is formal and singular at 0, and it is the unique that satisfies $\phi_\ast\wt{X}=X$. Since $\G$ is invariant for
$X$, we get $X|_{\g(s)}=h(s)\g'(s)$ for some $h(s)\in\C[[s]]$ and one obtains that $\wt X|_{\wt\g(s)}=h(s)\wt{\g}'(s)$, proving that $\wt \G$ is invariant for $X$ and that $\nu(\wt X|_{\wt\G})=\nu(X|_\G)$.

Assume now that $\phi$ is a permissible $l$-ramification, written
in some coordinates $\zz$ as $\phi(\zz)=(z_1^l,z_2,...,z_n)$.
Consider a parametrization of $\G$ of the form
$\g(s)=(s,\g_2(s),...,\g_n(s))$ in these coordinates (recall that,
from the definition of permissible ramification, $\G$ is
non-singular). Then
\begin{equation}\label{eq:gamma-tilde-ramification}
\wt{\g}(s)=(s,\g_2(s^l),...,\g_n(s^l))\in\C[[s]]^n
\end{equation}
is a parametrization of a formal curve $\wt{\G}$ satisfying
$\phi^*\G\subset\wt{\G}$. Uniqueness of $\wt{\G}$ comes from the
property of $\G$ being non-singular: $\G$ is generated by the
series $z_j-\g_j(z_1)$ for $j=2,...,n$ and thus, if $\overline{\Gamma}$ is a formal curve such that
$\phi^*\G\subset\overline{\G}$, then $\wt{\gamma}(s)$ is a parametrization of $\overline{\Gamma}$.

On the other hand, being $Z=\{z_1=0\}$ invariant for $X$, if we
write $X=\sum_{i=1}^n a_i(\zz)\frac{\partial}{\partial z_i}$ then
we have $ a_1(\zz)=z_1\bar{a}_1(\zz)$, where $\bar{a}_1(\zz)$ is a
formal series. The (singular) formal vector field $\wt{X}$
defined by
$$
\wt{X}=\frac{z_1 \bar{a}_1(\phi(\zz))}{l}\frac{\partial}{\partial
z_1}+\sum_{i=2}^n a_i(\phi(\zz))\frac{\partial}{\partial z_i}
$$
satisfies $\phi_\ast\wt{X}=X$. Since $\G$ is invariant and not contained in the singular locus of $X$,  $X|_{\g(s)}=h(s)\g'(s)$ for some $h\in\C[[s]]$ with $\nu(h)\ge1$. 
We obtain that  $\wt \G$ is invariant for $\wt X$ and 
\[ \nu(\wt X|_{\wt\G}) = \nu(X|_\G) + (\nu(X|_\G) -1)(l-1) \geq \nu(X|_\G) \]
as a consequence of $\wt X|_{\wt\g(s)}=l^{-1}s^{1-l}h(s^l)\wt{\g}'(s)$.
\end{proof}

It is worth to notice that
Proposition~\ref{pro:permissible-X} remains true, except
for the uniqueness of the curve $\wt{\G}$ satisfying
$\phi^*\G\subset\wt{\G}$, if the condition of $\G$ being
non-singular in the definition of permissible ramification is
removed (consider, for example, the curve $\G=(y^2-x^3)$ at $(\C^2,0)$ and
$\phi(x,y)=(x^2,y)$
where we can choose $\tilde{\G} = (y-x^3)$ or
$\tilde{\G} = (y+x^3)$). 

\begin{remark}\label{rk:finite-jets}
Observe that the expression of the transform
of a vector field by a permissible transformation is
\emph{finitely determined} in the following sense. Let $\phi$ be a permissible
transformation for $(X,\G)$ with center $Z$. Then, for any $N\in\N$, there exists $N'\in\N$
such that, if $Y$ is another formal vector field for which $Z$ is invariant and $J_{N'}Y=J_{N'}X$ then $J_{N}\wt{Y}=J_{N}\wt{X}$, where $\wt{X},\wt{Y}$ are the transforms of $X,Y$ by $\phi$, respectively. Although we do not require $Y$ to have $\Gamma$ as an invariant curve, the transform $\wt{Y}$ is well defined in Proposition~\ref{pro:permissible-X} once we have that the center $Z$ of $\phi$ is invariant for $Y$.
\end{remark}

\subsection{Reduction of a vector field to Ramis-Sibuya form}\label{sec:reduction}
Let $X$ be a formal singular vector field at $(\C^n,0)$ and let $\G$ be a formal invariant
curve of $X$ not contained in the singular locus of $X$. In this section we show that the pair $(X,\G)$ can be reduced to Ramis-Sibuya form by permissible transformations.

A {\em sequence of permissible transformations for
$(X,\G)$} is a composition
$$
\Phi=\phi_l\circ\phi_{l-1}\circ\cdots\circ\phi_1:\Cd n\to\Cd n
$$
such that $\phi_1$ is a permissible transformation for $(X,\G)$
and, for $j=1,...,l-1$, $(X_j,\G_j)$ is the transform of
$(X_{j-1},\G_{j-1})$ by $\phi_j$ and $\phi_{j+1}$ is a permissible
transformation for $(X_j,\G_j)$. The last pair $(X_l,\G_l)$ will
be called the {\em transform} of $(X,\G)$ by $\Phi$.
We also define the {\em total divisor} of $\Phi$ as the set
$E_\Phi=(\phi_l\circ\phi_{l-1}\circ\cdots\circ\phi_2)^{-1}(E_{\phi_1})$,
which is a normal crossing divisor at $0\in\C^n$.

\begin{theorem}\label{th:X-RS-form}
Let $X$ be a formal singular vector field at $0\in\C^n$ and
let $\G$ be an invariant formal curve of $X$ not contained in the
singular locus of $X$. Then there exists a sequence $\Phi$ of
permissible transformations for $(X,\G)$ such that the transform of $(X,\G)$ by $\Phi$ is in Ramis-Sibuya form.
\end{theorem}
A composition $\Phi$ as in the statement will be called a {\em reduction} of $(X,\G)$ to RS-form.

\strut

Theorem~\ref{th:X-RS-form} is not a completely new result  
in the theory of reduction of singularities of
vector fields or in the theory of systems of meromorphic
ODEs with irregular singularity. It contains in particular a result of
``local uniformization'' of $X$ (i.e. reduction to
non-nilpotent linear part) along the valuation corresponding to $\G$ (see Cano et al \cite{Can-M-R,Can-Roc-Spi} or Panazzolo \cite{Pan} for more information). The particular expression of a vector field in RS-form, that requires more than a non-nilpotent linear part, is obtained, once we associate to $X$ a system of $n-1$
meromorphic ODEs after some initial blow-ups, from classical
results in the theory of ODEs, generically known as {\em Turrittin's Theorem}: see Turrittin \cite{Tur}, Wasov \cite{Was}, Balser \cite{Bal} or Barkatou \cite{Bar} (for
linear systems with formal coefficients), and Cano et al. \cite{Can-Mou-San2} (for related
statements for three-dimensional real vector fields). Since we
could not find a statement with the precise terms of
Theorem~\ref{th:X-RS-form} needed for our
purposes, we
devote the rest of this section to provide a self-contained proof.

\strut

Let us describe the situation after a punctual blow-up. Let
$(x,\yy)$ be coordinates such that $\{x=0\}$ is transversal to (the tangent line of) $\G$ and let $\phi:\Cd
n\to\Cd n$ be the blow-up of $\C^n$ centered at the origin (which is permissible for $(X,\G)$).
There is a constant vector $\xi\in\C^{n-1}$ so that $\phi$ is
written as 
\begin{equation}\label{eq:point-blow-up-coordinates}
\phi(x,\yy)=(x,x(\yy-\xi)).
\end{equation}
We obtain the following properties:

{\bf (a1)} The transform of $X$ by $\phi$ is written as
$\wt{X}=x^{\nu(X)-1}X'$ where
$E_\phi=\{x=0\}$ and $X'$ is a formal singular vector field. Thus $\nu(\wt{X})\geq\nu(X)$ and the origin is again a
permissible center for the transform $(\wt{X},\wt{\G})$ of
$(X,\G)$.

{\bf (a2)} When $\xi=0$, the exponent of $x$ increases at least a unit in any
monomial of the coefficient $\wt{X}(x)$ with positive degree in the $\yy$-variables
and in any monomial of the components of $\wt{X}(\yy)$ with
degree at least two in the $\yy$-variables, whereas the order of $\wt{X}(\yy)(x,0)$ decreases in a unit.

\subsection*{From pre-RS form to RS form} 
To prove Theorem~\ref{th:X-RS-form}, it will be sufficient to prove that there exists a sequence $\Psi$ of permissible transformations for $(X,\G)$ such that the transform $\wt{X}=\Psi^*X$ is written in some coordinates $(x,\yy)$ as
\begin{equation}\label{eq:pre-RS-form}
\wt{X}=x^{q+1}u(x,\yy)\frac{\partial}{\partial x}
+\left(B_0(x)+(D(x)+x^qC)\yy+O(x^{q+1}\yy)\right)
\frac{\partial}{\partial\yy},
\end{equation}
where $u(0,0)\neq 0$, $B_0(0)=0$ and the transformed curve $\wt{\G}=\Psi^*\G$, together with $q,D(x),C$ satisfy the conditions of Definition~\ref{def:X-RS-form}. 

Let us see how we can obtain RS-form from expression \eqref{eq:pre-RS-form}. Analogously as in Remark~\ref{rk:RS-form}(b), if $\g(x)=(x,\bar{\g}(x))$ is a parametrization of $\wt{\G}$ and $\nu(\bar{\g}(x))\ge l$ then $\nu(B_0(x))\ge l$. Thus, by a change of variables of the form $\yy=\tilde{\yy}+J_{2q+2}\bar{\g}(x)$ we may assume that $\nu(B_0(x))\ge 2q+3$ and the first $2q+2$ iterated tangents of $\wt{\G}$ and of $\{\yy=0\}$ coincide. Taking into account the properties (a1) and (a2) above about the effect of a permissible punctual blow-up, we have that the composition $\Phi$ of the blow-ups at the first $q+1$ iterated tangents of $\wt{\G}$ is written as $\Phi(x,\yy)=(x,x^{q+1}\yy)$ and the transform $\Phi^*\wt{X}$ is written as in (\ref{eq:pre-RS-form}) with the extra hypothesis
$$
u(x,\yy)=u(x,0)+O(x^{q+1}\yy),\;\;\nu(B_0(x))\ge q+2.
$$
Notice that the matrix $D(x)+x^qC$ has changed into $D(x)+x^q(C-(q+1)u(0,0)I_{n-1})$. If this matrix vanishes, we 
consider $\Phi(x,\yy)=(x,x^{q+2}\yy)$ and $\Phi^{*} \tilde{X}$ is in the form (\ref{eq:pre-RS-form}) in which 
$D(x)+x^qC$ changes into $- x^q u(0,0) I_{n-1}$ and $\nu(B_0(x))\ge q+1$.
It remains to show that we can obtain
$u(x,0)=u(0,0)+ b x^{\max (1,q)} + O(x^{q+1})$. The series $u(x,0)$ is already in the required form
for $q=0$.
For the case $q\ge 1$, it suffices to consider a polynomial change of coordinates of the form $x=x+P(x)$, with
$P(x)=a_2x^2+\cdots+a_{q}x^{q}$.
This is consequence of a classical result for one-dimensional vector fields: if $Y=x^{q+1}v(x)\partial_x$ is a vector field with $v(x)=v_0+v_1x+\cdots$, $v_0\ne 0$ and $q\ge 1$, we can kill all coefficients $v_1,...,v_{q-1}$ with a polynomial change of variables, tangent to identity and of degree at most $q$.

A pair $(\wt{X},\wt{\G})$ in the form \eqref{eq:pre-RS-form} will be called a pair in {\em pre-RS-form}. In the rest of this section, we prove, to finish Theorem~\ref{th:X-RS-form}, that any pair $(X,\G)$ can be reduced to pre-RS-form by means of a finite composition of permissible transformations.

\subsection*{Reduction to pre-RS form} 
First, performing the blow-ups at the infinitely
near points of $\G$ and by resolution of singularities of
curves (see \cite{Wal}), we can assume that $\G$ is non-singular. Moreover, using property (a1) above, there is
a system of  coordinates $(x,\yy)$ for which $\G$ is transversal to $\{x=0\}$ and such that
$X=x^{e}\bar{X}$, where $\bar{X}$ is not divisible by $x$ and
$e\geq\nu(X)-1$ (in particular $\{x=0\}$ is contained in the singular locus of
$X$ if $\nu(X)\ge 2$). Let $\g(x)=(x,\bar{\g}(x))\in\C[[x]]^n$ be a parametrization
of $\G$ in these coordinates and write
$$
\bar{X}=a(x,\yy)\frac{\partial}{\partial x}+
 {\bf b}(x,\yy)\frac{\partial}{\partial\yy},
$$
where $a(x,\yy)\in\C[[x,\yy]]$ and ${\bf b}(x,\yy)\in\C[[x,\yy]]^{n-1}$. Since $\G$ is invariant and not contained in the singular locus of $X$ we obtain that
 $a(\g(x))\neq 0$.

\subsubsection*{Case $\bar{X}$ not singular}
We analyze first the case where $\bar{X}$ is not singular at the origin. In this case, we have $e\ge 1$ and, since $\G$ is the unique formal solution of $\bar{X}$ at $0$ and it
is transversal to $\{x=0\}$, we must have $a(0)=\lambda\neq 0$. We may assume also that $\G$ is tangent to $\{\yy=0\}$. After
a new blow-up at the origin, and taking coordinates as in
(\ref{eq:point-blow-up-coordinates}), the transform of $X$ is
written as
$$
\wt{X}=x^{e-1}\left[x\left(\lambda+O(x)\right)
\frac{\partial}{\partial x}
+\left({-\lambda} \yy+O(x)\right)\frac{\partial}{\partial\yy}\right],
$$
which is in pre-RS-form \eqref{eq:pre-RS-form} with $q=e-1\ge 0$ and Poincar\'{e} rank $p=0$.

\subsubsection*{Case $\bar{X}$ singular}
Assume now that $\bar{X}$ is a singular formal vector field. Let $r$ be
the order of vector field $\overline{X}|_\G$, it is equal to the order of
the series $a(\g(x))$. Notice that $1\leq r<\infty$. As in Remark~\ref{rk:RS-form}(b), we can assume, up to a polynomial change of variables of the form 
$\phi_1 (x, \yy) = (x, \yy +J_N\bar{\g}(x))$,  
 that the order in $x$ of any component of $b(x,0)$ is at least $2r+2$.
Let $\phi$ be the composition of the permissible blow-ups with center at the first
$r+1$ iterated tangents of $\G$,
written as $\phi(x,\yy)=(x,x^{r+1}\yy)$. Taking into account the effect, stated in property (a2), of a blow-up in the order with respect to $x$ of the different monomials of the coefficients of $X$
and the invariance of $\overline{X}|_\G$ under blow-ups,
we conclude that, after the transformation $\phi$, the vector field $X$ may be written as
\begin{equation}\label{eq:systemEDOs-1}
X=x^e\left[x^{r}u(x,\yy)\frac{\partial}{\partial
x}+(c(x)+A(x)\yy+
x^{r}\Theta(x,\yy))\frac{\partial}{\partial\yy}\right]
\end{equation}
where  $e\geq 0, r\geq 1$, $u(0,0)\neq 0$, $\nu(c(x))\geq r+1$, $A(x)\in\mathcal{M}_{n-1}(\C[[x]])$ and
$\Theta\in\C[[x,\yy]]^{n-1}$ has order at least 2 in the
$\yy$-variables.  Moreover, we may assume also that $A(0)\neq 0$: if $\nu(A(x))\ge r$ then $x^{-(e+r)}X$ is non-singular, a case already treated above; otherwise, if $\nu(A(x))<r$ we may rewrite $X$ as in (\ref{eq:systemEDOs-1}) replacing $e$ by $e+\nu(A(x))$ and $r$ by $r-\nu(A(x))$ so that the new matrix $A(x)$ satisfies $A(0)\ne 0$.

Put $r=s+1$ with $s\ge 0$. Notice that if $s=0$ then $X$ is already in the required pre-RS-form
(\ref{eq:pre-RS-form}) with $q=e$ and Poincar\'{e} rank $p=0$. We assume that $s\geq 1$. To
the vector field $X$ in (\ref{eq:systemEDOs-1}) we can associate the system of $n-1$ formal meromorphic ODEs
\begin{equation}\label{eq:systemEDOs-2}
x^{s+1}\yy'=u(x,\yy)^{-1}\left(c(x)+A(x)\yy+ x^r\Theta(x,\yy)
\right).
\end{equation}
We will use the following
classical result, that we state more or less as it appears in the book of Wasov \cite{Was}.
\begin{theorem}[Turrittin]\label{th:turritin}
Consider an $m$-dimensional system of formal linear ODEs $$x^{s+1}\ww'=\Lambda(x)\ww,\;\;\; \Lambda(x)\in\mathcal{M}_{m}(\C[[x]]),$$  and assume that $s\ge 1$ and $\Lambda(0)\ne 0$. Then, after a finite number of
transformations of the following types
\begin{enumerate}[$\bullet$]
\item Polynomial regular transformation
  $$
 L_{P(x)}(x,\ww)=(x,P(x)\ww),\;
 P(x)\in\mc{M}_{m}(\C[x])\mbox{ with }P(0)\mbox{ invertible}.
$$
  \item Shearing transformation
  $$
 S_{(k_1,...,k_m)}(x,\ww)=(x,\diag(x^{k_1},...,x^{k_m})\ww),\;k_j\in\N_{\geq 0}.
  $$
  \item Ramification
  $$
R_\alpha(x,\ww)=(x^\alpha,\ww),\; \alpha\in\N_{>0}.
  $$
\end{enumerate}
the system transforms into a system
$$x^{p+1}\ww'=(\overline{D}(x)+x^p\overline{C}+O(x^{p+1}))\ww,
$$ where either $p=0$ and $\overline{C}\neq0$ or $p\geq 1$, $\overline{D}(x)$ is a diagonal matrix of
polynomials
      of degree at most $p-1$ commuting with $\overline{C}$ and $\overline{D}(0)\neq 0$.
\end{theorem}
Polynomial regular transformations, shearing transformations and
ramifications, as defined in Turrittin's Theorem, will be called
{\em T-transformations}. Except for the ramifications, they are particular examples of {\em gauge} transformations of the system.

\begin{remark}\label{rk:shearing}
Note that a polynomial regular transformation does not change the number $s$ (the
{\em Poincar\'{e} rank} of the system) and that a ramification $R_\alpha$ multiplies
it by $\alpha$. The effect of a shearing transformation
$S_{(k_1,...,k_m)}$   on the Poincar\'{e} rank depends on the
parameters $k_1,...,k_m$ (and on the orders of the entries of the system). Looking carefully at the proof of
Theorem~\ref{th:turritin} (see for example \cite[section 19]{Was} or, alternatively, the proof in \cite{Bar}), we may observe that each shearing transformation in the process to prove Theorem~\ref{th:turritin} is chosen
so that its application never makes the Poincar\'{e} rank
increase strictly.
\end{remark}

We resume the proof of
Theorem~\ref{th:X-RS-form}. Assume that $X$ is
written as in (\ref{eq:systemEDOs-1}), that $\G$ is non-singular and transversal to $x=0$ and let $\g(s)=(s,\bar{\g}(s))$ be a parametrization of $\G$. Consider the formal change
of variables $ \yy=\hat{\yy}+\bar{\g}(x)$, for which
$\G=\{\hat{\yy}=0\}$, and write $X$ in the variables
$(x,\hat{\yy})$ as
$$x^{e}\left[x^{s+1}u(x,\hat{\yy}+\bar{\g}(x))\frac{\partial}{\partial
x}+(\hat{A}(x)\hat{\yy}+
x^{s+1}\hat{\Theta}(x,\hat{\yy}))\frac{\partial}{\partial\hat{\yy}}\right]$$
where $\hat{A}(0)\neq 0$ and
$\hat{\Theta}(x,\hat{\yy})=O(\|\hat{\yy}\|^2)$. The system
(\ref{eq:systemEDOs-2}) associated to the vector field $X$ becomes
$$x^{s+1}\hat{\yy}'=u(x,\hat{\yy}+\bar{\g}(x))^{-1}\left(\hat{A}(x)\hat{\yy}+
x^{s+1}\hat{\Theta}(x,\hat{\yy})\right).
$$
Apply Theorem~\ref{th:turritin} to the linear system
\begin{equation}\label{eq:linearY}
x^{s+1}\ww'=u(x,\bar{\g}(x))^{-1}\hat{A}(x)\ww,
\end{equation}
associated to the formal vector field
 $$Y=x^e\left[x^{s+1}u(x,\bar{\g}(x))\frac{\partial}{\partial
x}+\hat{A}(x)\ww\frac{\partial}{\partial\ww}\right].
$$
We get a composition $\Psi$ of $T$-transformations converting \eqref{eq:linearY} into a system with the prescribed properties stated in Theorem~\ref{th:turritin}. In terms of the associated vector field $Y$, if we write $\Psi(x,\ww)=(x^{\beta},\Psi_2(x,\ww))$, where $\beta$ is the product of the orders of the ramifications involved in the process and $\Psi_2$ is polynomial in $x$ and linear in $\ww$, we get  
\begin{equation}\label{eq:turrittin-to-linear}
\Psi^*Y=\beta^{-1}u(x^{\beta},\bar{\g}(x^{\beta})) x^{\nu}\left[x^{p+1}  \frac{\partial}{\partial
x}+\left(\overline{D}(x)+x^p\overline{C}+O(x^{p+1})\right)\ww\frac{\partial}{\partial\ww}\right]
\end{equation}
where $\nu\geq 0$ and $p$, $\overline{D}(x)$ and $\overline{C}$ satisfy the properties stated
in Theorem~\ref{th:turritin}. In fact, we have $p+\nu=\beta(e+s)$.

\begin{lemma}
$\Psi$ is  a sequence of
permissible transformations for $(Y,\{\ww=0\})$.
\end{lemma}
\begin{proof}
This is clear for
polynomial regular transformations and for ramifications. On the
other hand, a shearing transformation can be viewed as a
composition of blow-ups. More
precisely, consider $\phi=S_{(k_2,...,k_{n})}$ where $k_j >0$ if $2 \leq j \leq t$ and $k_j=0$
otherwise. The expression of the shearing transformation is
$\phi (x, \ww)=(x, x^{k_2} w_2, \hdots, x^{k_t} w_t, w_{t+1}, \hdots, w_n)$.  
Put
\[ \overline{Y}=x^{-e}Y =  {a}_1(x,\ww)\frac{\partial}{\partial x} +
\sum_{i=2}^n {a}_i(x, \ww)\frac{\partial}{\partial w_i}. \]
We obtain $\phi^{*} \overline{Y}= \tilde{a}_1(x,\ww)\frac{\partial}{\partial x} +
\sum_{i=2}^n
\tilde{a}_i(x, \ww)\frac{\partial}{\partial w_i}$ defined by
\[
\left\{
  \begin{array}{ll}
    \tilde{a}_j(x,\ww)=\dfrac{a_j(\phi(x, \ww))-k_j x^{k_j -1} w_j a_1(\phi(x,\ww))}{x^{k_j}}, &
 \hbox{for j=2,...,t;} \\[7pt]
    \tilde{a}_j(x,\ww)=a_j(\phi(x,\ww)), & \hbox{for $j \in \{1,t+1, \hdots, n\}$.}
  \end{array}
\right.
\]
Since the Poincar\'e rank does not increase by  the shearing transformations in Turrittin's process (see Remark~\ref{rk:shearing}), the pull-back $\phi^*\overline{Y}$ has coefficients in $\C[[x,\ww]]$ (with no poles). We deduce
that $a_1, \hdots, a_t$ belong to the ideal $(x,w_2, \hdots, w_t)$.
Therefore $Z=\{x=w_2=\hdots=w_t=0\}$ is invariant by $\overline{Y}$ and the
blow-up of $Z$ is a permissible transformation.
Then we consider the blow-up of $\{x=0\} \cap \cap_{k_j  \geq 2} \{ w_j=0\}$. Analogously as
above, it is a permissible transformation. By repeating this process with centers of
the form $\{x=0\} \cap \cap_{k_j  \geq  d} \{ w_j=0\}$ for
$1 \leq d \leq \max(k_1, \hdots, k_t)$,
we get that any shearing transformation  is a sequence of permissible transformations
for $(\overline{Y},\{\ww=0\})$, and hence for $(Y,\{\ww=0\})$.
\end{proof}

Notice that the pair $(\Psi^*Y,\{\ww=0\})$, where $\Psi^*Y$ is given in \eqref{eq:turrittin-to-linear}, is in pre-RS-form.
Let us show how to reduce $(X,\G)$ to pre-RS-form from this property. For any $m\ge 1$, consider the diffeomorphism $\phi_m:(\C^n,0)\to(\C^n,0)$
 given by $\phi_m(x,\yy)=(x,\yy-J_{2m} \bar{\g}(x))$.  
 The transform $X_m=\phi_m^*X$ is written as
$$
X_m=x^{e}\left[x^{s+1}u(x,\yy+J_{2m} \bar{\g}(x))\frac{\partial}{\partial
x}+\left(c_m(x)+A_m(x)\yy+
x^{s+1}\Theta_m(x,\yy)\right)\frac{\partial}{\partial\yy}\right],
$$
where $A_m(0)\ne 0$, $\Theta_m(x,\yy)=O(\|\yy\|^2)$ and $c_m(x),A_m(x),\Theta_m(x,\yy)$ converge respectively to $0,\hat{A}(x),
\hat\Theta(x,\yy)$ in the Krull topology associated to the
ideal $(x)$ (also called the $(x)$-adic topology) when $m \to \infty$.
Moreover, the transform $\G_m=\phi_m^*\G$ has a parametrization
$(x, \bar{\g} (x) - J_{2m} \bar{\g} (x))$ that converges to
$(x, 0)$ in the $(x)$-adic topology. Therefore, we get $\nu(c_m(x))\ge 2m+1$
(see Remark~\ref{rk:RS-form}).
Consider the map $\psi_m(x,\yy)=(x,x^m \yy)$,  
a composition of punctual permissible
blow-ups for $(X_m,\G_m)$. Let $(X_{m}',\G_{m}')$ be
the transform of $(X,\G)$ by $\psi_m \circ\phi_m$.

Taking into account formula (\ref{eq:X-after-blowup}) and property (a2) above, we get that
the limit of $X_{m}''=X_{m}'  + m W$,
where $W= x^{e+s}u(x,\bar{\g}(x)) \yy\frac{\partial}{\partial\yy}$, in
the $(x)$-adic topology is equal to $Y$ when $m \to \infty$.
Moreover, the parametrization
$(x, (\bar{\g}(x) - J_{2m} \bar{\g}(x))/x^{m})$ of
$(\psi_m \circ \phi_m)^{*} \Gamma$ converges to $(x,0)$ in the $(x)$-adic topology
when $m \to \infty$.
It is straightforward to check out that, since $\Psi$ is a sequence of permissible transformations
for $(Y,\{\ww=0\})$, there exists a neighborhood $U$ of  $Y$ in the
$(x)$-adic topology such that  $ \Psi^{*} Z$  is a formal vector field for any $Z \in U$ and
the map $Z \mapsto \Psi^{*} Z$ is continuous in $U$ where we consider the
 $(x)$-adic topology in both the source and the target.

In order to finish the proof of
Theorem~\ref{th:X-RS-form}, it suffices to
prove that if $m$ is sufficiently big then $\Psi$ reduces $(X_{m}',\G_{m}')$ to pre-RS-form:
the map
$\Psi\circ\psi_m \circ\phi_m$ will then reduce $(X,\G)$
to pre-RS-form.
Since $\lim_{m \to \infty}  X_{m}'' = Y$, $\Psi^{*}$ is continuous at $Y$ and
$\Psi^{*} (W) = \tau (x) \ww\frac{\partial}{\partial\ww}$
for   $\tau (x)=x^{\beta(e+s)} u(x^{\beta}, \bar{\g}(x^{\beta}))$,
we deduce that $\Psi$ is a permissible transformation for
$(X_m', \G_{m}')$. The continuity of $\Psi^{*}$ at $Y$ implies that
the pair $(\Psi^{*} X_{m}', \Psi^{*} \G_{m}')$ is in pre-RS-form
where the matrix $x^{\nu} (\overline{D}(x)+x^p\overline{C})$
in equation (\ref{eq:turrittin-to-linear}) is replaced by 
\[ x^{\nu} \overline{D}(x)  +
x^{p+ \nu} (\overline{C} - m c I_{n-1}) \]
where  $J_{p+ \nu} \tau (x)= c x^{p+\nu}$ (recall that $\tau(x)$ has order $\beta(e+s)=p+\nu$).  Indeed,
the above matrix satisfies the conditions in Definition \ref{def:X-RS-form} for $m >>1$.
Theorem~\ref{th:X-RS-form} is finished.

\subsection{Reduction of a biholomorphism to Ramis-Sibuya form}\label{sec:RS-diffeos}

Consider a biholomorphism $F\in\Diff\Cd n$ having a formal invariant curve $\G$.
In this section we use Theorem~\ref{th:X-RS-form} and the results in Section~\ref{sec:infinitesimal} to obtain, up to iteration of $F$, a reduction of the pair $(F,\G)$ to a form analogous to the Ramis-Sibuya form in Definition~\ref{def:X-RS-form}.

First, we need to adapt Proposition~\ref{pro:permissible-X} to the context of flows.

%
\begin{proposition}\label{pro:permissible-F}
Consider $F\in\Diff\Cd n$ with a formal invariant curve $\G$ and assume that there exists a formal vector field $X\in\cvf\Cd n$ such that $F=\Exp X$ and $\G$ is invariant for $X$. Let $\phi:\Cd n\to\Cd n$ be a permissible transformation for
$(X,\G)$ and let $(\wt{X},\wt{\G})$ be the transform of $(X,\G)$ by $\phi$. Then the diffeomorphism $\wt{F}=\Exp\wt X$ is analytic,
satisfies
  $\phi\circ\wt{F}=F\circ\phi$ and has $\wt{\G}$ as invariant curve.
We say that $\phi$ is a {\em permissible transformation} for $(F,\G)$ and that the pair $(\wt{F},\wt{\G})$ is the {\em transform} of
$(F,\G)$ by $\phi$.
\end{proposition}
\begin{proof}
If $\phi$ is a germ of a diffeomorphism, the result is clear.

Assume that $\phi$ is a permissible blow-up with center $Z$.
Consider analytic coordinates $\zz=(z_1,z_2,...,z_n)$ such that $\G$
is tangent to the $z_1$-axis, $Z=\{z_1=z_2=\cdots=z_t=0\}$ and
$\phi(\zz)=(z_1,z_1z_2,...,z_1z_t,z_{t+1},...,z_n)$. The
condition $\phi\circ\wt{F}=F\circ\phi$ can be written as
$$
z_j\circ\wt{F}(\zz)=\left\{%
\begin{array}{ll}
    \dfrac{z_j\circ F (\phi(\zz))}{z_1\circ F (\phi(\zz))}, & \hbox{if $j=2,...,t$;}  \\[10pt]
    z_j\circ F(\phi(\zz)), & \hbox{if $j\in\{1,t+1,...,n\}$}. \\
\end{array}%
\right.
$$
Since the tangent line of $\G$ is invariant for $D_0F$, the series $z_1\circ F(\zz)$ has a monomial of the form $az_1$, with $a\neq0$. Moreover, since $Z$ is invariant for $F$, we have $z_j\circ F(\zz)\in(z_1,...,z_t)$ for $j=1,...,t$. Therefore $z_1\circ F(\phi(\zz))=z_1(a+A(\zz))$ with $A(0)=0$ and $z_j\circ F(\phi(\zz))$ is divisible by $z_1$
for $j=2,...,t$. We conclude that $\wt{F}\in\Diff\Cd n$. By construction, $\wt{F}$ is the unique formal diffeomorphism such that $\phi\circ\wt{F}=F\circ\phi$. We have also that $\wt F=\Exp\wt{X}$ since $\wt{X}^j(g\circ\phi)=X^j(g)\circ\phi$
for any $g\in\hat{\mathcal{O}}_n$ and any $j\geq 1$, and thus $\Exp\wt{X}$ also satisfies
formally $\phi\circ\Exp\wt{X}=F\circ\phi$. Notice finally that $\wt\G$ is invariant for $\wt X$ by Proposition~\ref{pro:permissible-X} and hence $\wt\G$ is also invariant for $\wt F=\Exp\wt X$.

Assume now that $\phi$ is a permissible $l$-ramification, written
in some coordinates $\zz$ as $\phi(\zz)=(z_1^l,z_2,...,z_n)$, that is, $Z=\{z_1=0\}$ is the center of $\phi$.
As in the case of permissible blow-ups, we have that the formal diffeomorphism $\wt F=\Exp\wt X$ satisfies $\phi\circ\wt F=F\circ\phi$. This identity means
$$
(z_1\circ\wt{F}(\zz))^l=z_1\circ
F(\phi(\zz));\quad z_j\circ\wt{F}(\zz)=
   z_j\circ F (\phi(\zz)),\,j=2,...,n.
$$
On the other hand, since $Z$ is invariant for $F$, we have $z_1\circ F(\zz)=z_1(a+A(\zz))$, with $a\neq0$ and $A(0)=0$. We conclude that
$\wt{F}\in\Diff\Cd n$. The proof of the invariance of $\wt\G$ by $\wt{F}$ is the same as in the previous case.
\end{proof}
\begin{remark}
\label{rem:inv_inf_gen}
Assume that $X$ is an infinitesimal generator of $F$ in the sense of Definition~\ref{def:infgen}. Then, the transformed vector field $\wt{X}$ in Proposition \ref{pro:permissible-F} is an infinitesimal generator of $\wt{F}$, since being not weakly
resonant is invariant by $T$-transformations. More precisely, given $X \in \cvf\Cd n$ with
$\mathrm{spec} (D_0 X)= \{ \mu_1, \hdots, \mu_n \}$, consider
\[ R(X)= \{ m_1 \mu_1 + \hdots + m_n \mu_n : m_1, \hdots, m_n \in {\mathbb Q} \} . \]
It is easy to verify that $R(X) = R(\wt{X})$ for regular transformations,
ramifications and blow-ups. It also holds for  shearing transformations since they
are compositions of blow-ups.
\end{remark}

The main result in this section is the following.

\begin{theorem}[Reduction of a diffeomorphism to Ramis-Sibuya form]\label{th:F-RS-form}
Let $F\in\Diff\Cd n$ be a germ of a diffeomorphism having a formal invariant curve $\G$. Assume that $\G$ is rationally neutral and not contained in the set of fixed points of any non-trivial iterate of $F$. Let $m$ be the index of embeddability of $F$. Then there exists a finite composition $\Phi$ of permissible transformations
for $(F^m,\G)$ and some coordinates $(x,\yy)$ at $0\in\C\times\C^{n-1}$ so that, if $(\wt{F^m},\wt\G)$ is the transform of $(F^m,\G)$ by $\Phi$, then $\wt{\G}$ is non-singular and transversal to $\{x=0\}$ and $\wt{F^m}$ is written as:
\begin{equation}\label{eq:RS-form}
\left\{
\begin{array}{l}
  x\circ\wt{F^m}(x,\yy)=x-x^{q+1}+  b x^{2q+1} + O(x^{2q+2}) \\
  \yy\circ\wt{F^m}(x,\yy)=\exp(D(x)+x^qC)\yy+O(x^{q+1}),
\end{array}\right.
\end{equation}
where $q\ge 1$,  $b \in {\mathbb C}$,
$D(x)$ is a diagonal matrix of polynomials of
degree at most $q-1$, $C$ is a constant matrix, $D(x)+x^qC\not\equiv 0$, $[D(x),C]=0$ and the order of contact of $\wt{F^m}$ with the identity  coincides with the order of the matrix $D(x)+x^qC$ plus one. In this case, we say that the pair $(\wt{F^m},\wt\G)$ is in \emph{Ramis-Sibuya form}.
\end{theorem}
\begin{proof}
By Theorem~\ref{theorem:infinitesimal_generator}, the iterate $F^m$ has an infinitesimal generator $X\in\cvf\Cd n$.
By Proposition~\ref{pro:remres2}, $\Gamma$ is invariant for $X$
and $\nu(X|_\Gamma)\ge2$
(notice that $\G$ is rationally neutral for any iterate of $F$).  Moreover, $\G$ is not contained in the singular locus of $X$ since, otherwise, $\G$ would be contained in the set of fixed points of $F^m=\Exp X$. By Theorem~\ref{th:X-RS-form}, there exists a composition $\Phi$ of finitely many permissible transformations for $(X,\G)$  such
that the transform $(\wt{X},\wt{\G})$ of $(X,\G)$ by $\Phi$ is in RS-form. Fix some coordinates
$(x,\yy)$ such that $(\wt{X},\wt{\G})$ is written as in equation (\ref{eq:X-RS-form}):
$$
\wt{X}=\left(\lambda x^{q+1}+  b x^{2q+1} + O(x^{2q+2})
\right)\parcial x+\left((D(x)+x^qC)\yy+O(x^{q+1})\right)\parcial\yy.
$$
Notice that $q\ge 1$ since $\nu(\wt X|_{\wt\G})\ge2$ by Proposition~\ref{pro:permissible-X}.
In particular, by Remark~\ref{rk:RS-form} we may assume that $\lambda=-1$.
We conclude that the transform $\wt{F^m}=\Exp\wt X$ of $F$ by $\Phi$ is written as in equation (\ref{eq:RS-form}) with the required properties $q\ge 1$, $D(x)+x^qC\not\equiv 0$, $D(x)$ diagonal of degree at most $q-1$ and $[D(x),C]=0$. Let $\nu$ be the order of $D(x)+x^qC$. It remains to prove that the vector $\yy\circ\wt{F^m}-\yy\in\C\{x,\yy\}^{n-1}$ has order $\nu+1$. If $\wt{F^m}$ is not tangent to the identity then $D_\yy\wt{F^m}(0)=\exp((D(x)+x^qC))|_{x=0}\ne I_{n-1}$, which implies $\nu=0$ and the property holds. Suppose that $\wt{F^m}$ is tangent to the identity.
Then, by Remarks ~\ref{rem:unipotentdiffeo} and \ref{rem:inv_inf_gen}, $\wt{X}$ is the unique nilpotent vector field such that $\wt{F^m}=\exp\wt{X}$ and we know that $\nu(\wt X)\ge2$
since $I=D_0 \wt{F^m}= \Exp (D_0 \wt X)$ and $D_0 \wt X$ is nilpotent.
Using the formula for the exponential, we conclude that $\yy\circ\wt{F^m}-\yy$ has order equal to $\nu(\wt{X})=\nu+1$, as wanted.
 \end{proof}

To finish this section, we show that stable manifolds of $F$, as well as asymptotic orbits to $\G$, are preserved under a permissible transformation for $(F,\G)$. Together with Theorem~\ref{th:F-RS-form}, this allows us to assume that the pair $(F,\G)$ is in Ramis-Sibuya form in order to prove Theorem~\ref{th:main} in the (remaining) case where $\G$ is rationally neutral and not contained in the set of fixed points of any iterate of $F$. Recall that to obtain Ramis-Sibuya form, we have first considered an iterate of $F$, then performed some punctual blow-ups along $\G$ and, once the transform of $\G$ is non-singular, some other permissible blow-ups with a center with a dimension possibly greater than $0$ and ramifications. Then, for our purposes, it is sufficient to consider the statement in the following way.

\begin{proposition}\label{pro:permissible-orbits}
Let $\phi$ be a permissible transformation for $(F,\G)$ with center $Z$ and exceptional divisor $E_\phi=\phi^{-1}(Z)$. Assume that $\G$ is non-singular if the center $Z$ has positive dimension. Consider a representative $\phi:\wt{V}\to V$ such that $F$ is defined in $V$ and $Z$ is an analytic smooth subvariety of $V$. Let $(\wt{F},\wt{\G})$ be the transform of $(F,\G)$ by $\phi$. We have
\begin{enumerate}[(i)]
\item If $\wt{S}\subset\wt{V}$ is a stable manifold of $\wt{F}$ in $\wt{V}$ such that $\wt{S}\cap E_\phi=\emptyset$ then $S=\phi(\wt{S})$ is a stable manifold of $F$ in $V$. Moreover, if $\wt{O}\subset\wt{S}$ is a $\wt{F}$-orbit asymptotic to $\wt{\G}$ then $O=\phi(\wt{O})$ is a $F$-orbit asymptotic to $\G$.
\item If $S\subset V$ is a stable manifold of $F$ such that $S\cap Z=\emptyset$ and every $F$-orbit in $S$ is tangent to $\G$, then $\wt{S}=\phi^{-1}(S)$ is a stable manifold of $\wt{F}$. Moreover, if $O\subset S$ is a $F$-orbit asymptotic to $\G$ then $\wt{O}=\phi^{-1}(O)$ is a $\wt{F}$-orbit asymptotic to $\wt{\G}$.
\end{enumerate}
\end{proposition}
\begin{proof}
The two assertions concerning the stable manifolds are consequences of the fact that $\phi\circ\tilde F=F\circ\phi$, together with the fact that $\phi$ is an isomorphism outside the divisor $E_\phi$. The assertions concerning the asymptoticity of the orbits are immediate from the definition in the case where $\phi$ is the blow-up at $0$. In the other cases, we take coordinates $\zz$ such that $\G$ is parameterized by $\g(s)=(s,\g_2(s),...,\g_n(s))$ and such that $\phi$ is either a ramification with respect to $Z=\{z_1=0\}$ or is written as in (\ref{eq:expression-blow-up}) in the case of a blow-up. Using the corresponding formulas (\ref{eq:gamma-tilde-ramification}) or (\ref{eq:gamma-tilde}) for a ramification of the transformed curve $\wt{\G}$ (again non-singular), the result is a consequence of the characterization of asymptoticity of orbits to a non-singular curve in terms of a parametrization of the curve (see Section~\ref{sec:blow-ups}).
\end{proof}

\section{Existence of stable manifolds}\label{sec:stable-manifolds}
Consider a diffeomorphism $F\in\Diff\Cd n$ and a formal non-singular invariant curve $\G$ such that the pair $(F,\G)$ is in Ramis-Sibuya form, i.e. there exist coordinates $(x,\yy)=(x,y_2,...,y_n)$ at $0\in\C^n$ such that $\G$ is transverse to $x=0$ and such that $F$ is written as
\begin{align*}
  x\circ F(x,\yy)&=x-x^{q+1}+  b x^{2q+1} + O(x^{2q+2}) \\
  \yy\circ F(x,\yy)&=\exp\left(D(x)+x^qC\right)\yy+O(x^{q+1}),
\end{align*}
where $q\ge1$,  $b \in{\mathbb C}$
and $D(x)$ and $C$ satisfy the properties of Theorem~\ref{th:F-RS-form}. Denote by $k+1$ the order of contact of $F$ with the identity, which coincides with the order of $D(x)+x^qC$ plus one. Note that $0\le k\le q$, and put $p=q-k\ge0$.

We define the \emph{attracting directions} of $(F,\G)$ as the $q=k+p$ half-lines $\{x\in\xi\R^+\}$, where $\xi^{k+p}=1$. Observe that, when $\G$ is convergent, these directions are the
limits of the secant real lines passing through the origin and points in an orbit of
the restricted diffeomorphism $F|_{\G}\in\Diff(\C,0)$, converging to $0$.
We classify the attracting directions of $(F,\G)$  as follows. Write $D(x)+x^qC=x^k\left(\ol D(x)+x^pC\right)$, where $\ol D(x)=0$ in case $p=0$. In case $p\ge1$, set
$$\ol D(x)=\diag(d_2(x),...,d_n(x))$$
and, for any
$2\le j\le n$, write
$d_j(x)=A_{j,\nu_j}x^{\nu_j}+A_{j,\nu_j+1}x^{\nu_j+1}+\dots+A_{j,p-1}x^{p-1}$ if $d_j(x) \neq 0$, 
where $\nu_j$ is the order of $d_j$ at 0. Given an attracting direction $\ell=\xi\R^+$ and $j\in\{2,...,n\}$, we say that  $\ell$ is a \emph{node direction} for $(F,\G)$ in the variable $y_j$ if $p\ge1$, $d_j(x)\neq0$ and
$$\left(\Real\left(\xi^{k+\nu_j}A_{j,\nu_j}\right), \Real\left(\xi^{k+\nu_j+1}A_{j,\nu_j+1}\right), ...,\Real\left(\xi^{k+p-1}A_{j,p-1}\right)\right)<0$$
in the lexicographic order; otherwise, we say that it is a \emph{saddle direction} for $(F,\G)$ in the variable $y_j$. Note that, if $p=0$, any attracting direction is a saddle direction in every variable. 

The rest of this section is devoted to complete the proof of  
Theorem~\ref{th:main}. After the results in Section~\ref{sec:hyperbolic} and Theorem~\ref{th:F-RS-form}, it suffices to show the following theorem.

\begin{theorem}\label{th:stablemanifold}
Consider a pair $(F,\G)$ in Ramis-Sibuya form and let $\ell$ be an attracting direction of $(F,\G)$. Let $s-1\ge 0$ be the number of variables for which $\ell$ is a node direction. Then, there exists a stable manifold $\mathcal{S}_{\ell}$ of $F$ of dimension $s$ in which every orbit is asymptotic to $\G$ and tangent to $\ell$. More precisely, there exist a connected and simply connected domain $S\subset\C^s$ with $0\in\partial S$ and a holomorphic map $\varphi:S\to \C^{n-s}$ such that, up to reordering the variables, the set
$$\mathcal{S}_{\ell}=\left\{\left(x,\ww,\varphi(x,\ww)\right)\in\C\times\C^{s-1}\times\C^{n-s}:(x,\ww)\in S\right\}$$
satisfies the following properties:
\begin{enumerate}[i)]
	\item $\mathcal{S}_\ell$ is a stable manifold of $F$.
	\item Every orbit $\{(x_j,\yy_j)\}\subset\mathcal{S}_\ell$ is asymptotic to $\G$ and $\{x_j\}$ is tangent to $\ell$.
	\item If $\{(x_j,\yy_j)\}\subset \C\times\C^{n-1}$ is an orbit of $F$ asymptotic to $\G$ such that $\{x_j\}$ has $\ell$ as tangent direction, then $(x_j,\yy_j)\in \mathcal S_\ell$ for all $j$ sufficiently big.
\end{enumerate}
\end{theorem}

\subsection*{Choice of coordinates} Up to a linear change of coordinates in the $x$-variable, we may assume that $\ell=\R^+$. We can also assume, without loss of generality, that $\ell$ is a node direction in the variables $y_2,...,y_s$ and a saddle direction in the variables $y_{s+1},...,y_n$ and that $C$ is in Jordan normal form (see Remark~\ref{rk:RS-form}).

Observe that we can increase the order of contact of $\G$ with the $x$-axis
by considering a polynomial change of variables of the form
$(x,\yy) \mapsto (x, \yy- J_N \ol\g(x))$ where $\g(x)=(x,\ol\g(x))$
is a parametrization of $\G$. Moreover, the matrices $D(x)$ and $C$ that appear in the expression of $\yy\circ F$ are preserved by such transformations.
Note
also that after a permissible punctual blow-up the transformed pair
$(\widetilde{F}, \widetilde{\G})$ is again in Ramis-Sibuya form in usual coordinates
$(x,\ol\yy)$ with
$\yy = x \ol\yy$ as in Section~\ref{sec:reduction}.
Moreover, the matrix $D(x)$ is invariant
by blow-up whereas $C$ is replaced by $C + I_{n-1}$.
Consequently, the saddle or node character of $\ell=\R^+$ in each variable does not change and, by Proposition~\ref{pro:permissible-orbits}, it suffices to prove Theorem~\ref{th:stablemanifold} in the new
coordinates $(x,\overline{\yy})$.
 Therefore, taking $N$ sufficiently big and up to several punctual admissible blow-ups,  if we put $(x,\yy)=(x,\ww,\zz)\in\C\times\C^{s-1}\times\C^{n-s}$ we can write $F$ as
\begin{align*}
x\circ F(x,\yy)&=f(x,\yy)=x-x^{k+p+1}+  b x^{2k+2p+1} + O(x^{2k+2p+2}) \\
\ww\circ F(x,\yy)&=\overline{F}_1(x,\yy)=\exp\left(x^k\left(\ol D_1(x)+x^pC_1\right)\right)\ww+O(x^{k+p+1})\\
\zz\circ F(x,\yy)&=\overline{F}_2(x,\yy)=\exp\left(x^k\left(\ol D_2(x)+x^pC_2\right)\right)\zz+O(x^{k+p+1}),
\end{align*}
where $b \in {\mathbb C}$,
$\ol D_1(x), \ol D_2(x), C_1, C_2$ are the corresponding blocks of $\ol D(x)$ and $C$ (note that this decomposition is guaranteed by the commutativity of $D(x)$ and $C$, see Remark~\ref{rk:RS-form}) and every eigenvalue of $C_2$ has positive real part.

In fact, we will use coordinates for which $\G$ has an arbitrarily big order of contact with the $x$-axis. Fix $m\in\N$, with $m\ge p+2$, and let $\g(x)=(x,\overline{\g}(x))$ be a parametrization of $\G$. Consider the polynomial change of variables $\yy\mapsto \yy^m=\yy-J_{p+m-1}\overline{\g}(x)$. In these coordinates, the order of contact of $\G$ with the $x$-axis is at least $p+m$, and the invariance of $\G$ implies that the order of $\yy^m\circ F(x,0)$ is at least $k+p+m$. Therefore, if we set $(x,\yy^m)=(x,\ww^m,\zz^m)\in\C\times\C^{s-1}\times\C^{n-s}$ we have
\begin{align*}
f(x,\yy^m)&=x-x^{k+p+1}+  b x^{2k+2p+1} + O(x^{2k+2p+2}) \\[4pt]
\ol F_1(x,\yy^m)&=\exp\left(x^k\left(\ol D_1(x)+x^pC_1\right)\right)\ww^m+O(x^{k+p+1}\|\yy^m\|,x^{k+p+m})\\[4pt]
\ol F_2(x,\yy^m)&=\exp\left(x^k\left(\ol D_2(x)+x^pC_2\right)\right)\zz^m+O(x^{k+p+1}\|\yy^m\|,x^{k+p+m}).
\end{align*}

Write, as above, $\ol D(x)=\diag(d_2(x),...,d_n(x))$, where $d_j(x)$ is a polynomial of degree at most $p-1$ for all $2\le j\le n$, and set $C_2=\diag(A_{s+1,p},...,A_{n,p})+N_2$, where $A_{j,p}\in\C$ for all $s+1\le j\le n$ and $N_2$ is a nilpotent matrix.

For any $2\le j\le n$, write
\[ e_{j}(x) := d_j(x) +  A_{j,p} x^{p} =A_{j,\nu_j}x^{\nu_j}+A_{j,\nu_j+1}x^{\nu_j+1}+\dots+A_{j,p}x^{p}, \]
where $\nu_j$ is the order of $e_j$ at 0 and $A_{j,p}=0$ for $2\le j\le s$. 
We set $k+\mu_j$ as the order of $x^{k} e_j(x)  - i \Imag [(x^{k} e_j(x))(0)]$. Note that $\mu_j=\nu_j$ if $k+\nu_j\ge1$ or if $k+\nu_j=0$ and $\Real(A_{j,0})\neq0$. Therefore, we have that
$$\Real\left(x^k e_j(x)  \right)=\Real\left(A_{j,\mu_j}x^{k+\mu_j}+A_{j,\mu_j+1}x^{k+\mu_j+1}+\dots+A_{j,p}x^{k+p}\right).$$
We define the \emph{first asymptotic significant order} $r_j=r_j(\ell)$ 
of $\ell=\R^+$ in the variable $y_j$ as $r_j=l-\mu_j$, where $\mu_j\le l \le p$ is the first index such that $\Real(A_{j,l})\neq0$ (note that $l$ is well defined because of the previous condition on the eigenvalues of $C_2$).

Note that $r_j$ does not depend on $m$, and that $r_j<k+p$ for all $j$: the inequality is clear if $p=0$ or $k+\mu_j\ge1$; otherwise, $r_j=0$ so it also holds.

Put $r=\max\{r_2,...,r_n\}$. For $d,e,\varepsilon>0$, we define the set $R_{d,e,\varepsilon}$ as
$$R_{d,e,\varepsilon}=\{x\in\C:|x|<\varepsilon, \Real x>0, -d(\Real x)^{r+1}<\Imag x<e(\Real x)^{r+1}\}.$$

\begin{lemma}\label{lem:nodaldomain}
Set $t=\max\{r_2+\mu_2,...,r_s+\mu_s\}<p$. There exists a constant $c>0$ such that, if $d,e,\varepsilon>0$ are sufficiently small, then for any $x\in\petal$ we have  
\begin{enumerate}[(i)]
	\item $\Real\left(x^kd_j(x)\right)\le -c|x|^{k+t}$
	for any $2\le j\le s$.
	\item $\Real\left(x^kd_j(x)+x^{k+p}A_{j,p}\right)\ge c|x|^{k+p}$
	for any $s+1\le j\le n$.
\end{enumerate}

\end{lemma}
\begin{proof}
Let us prove (i), the proof of (ii) is analogous.  Fix $2\le j\le s$. If $r_j=0$, we have that 
$$\Real(x^kd_j(x))\le \Real(A_{j,\mu_j}x^{k+\mu_j})/2\le -c_j|x|^{k+\mu_j}$$
if $d,e,\varepsilon$ are sufficiently small, where $-c_j=\Real(A_{j,\mu_j})/3$. If $r_j\ge 1$, we have that $k+\mu_j\ge1$ and we use the same argument of \cite[Lemma 5.9]{Lop-R-R-S}, that we include for the sake of completeness. We assume that $\Imag(A_{j,\mu_j})>0$; the other case is analogous. Using indeterminate coefficients we can see that there exists a diffeomorphism $\rho(x)=x+\sum_{l\ge2}\rho_lx^l$ such that
$$A_{j,\mu_j}x^{k+\mu_j}+\dots+A_{j,p-1}x^{k+p-1}=A_{j,\mu_j}\rho(x)^{k+\mu_j},$$
with $\rho_l\in\R$ if $2\le l\le r_j$ and $\Imag(\rho_{r_j+1})>0$.
Hence, to prove the inequality in (i),   it suffices to show that $\Real(A_{j,\mu_j}x^{k+\mu_j})\le -c_j|x|^{k+r_j+\mu_j}$ for some $c_j>0$ and for all $x\in\rho(\petal)$.
It is easy to show, using the fact that $\rho_l\in\R$ for $2\le l\le r_j$, that for any $a\in\R$, the image under $\rho$ of the curve 
$$\Imag x=a(\Real x)^{r_j+1}$$ is a curve of the form
$$C_a:\;\; \Imag x=(a+\Imag(\rho_{r_j+1}))(\Real x)^{r_j+1}+\dots$$
Then, since $r\ge r_j$, we obtain that set $\rho(\petal)$ is contained, if $\varepsilon$ is small enough, in a domain enclosed by two  curves of the type $C_e$ and $C_{-d}$. If $d$ is sufficiently small, then $-d+\Imag(\rho_{r_j+1})>0$, so $d'|x|^{r_j}<\arg x<\pi/(2(k+\mu_j))$ for some $d'>0$ and for all $x\in\rho(\petal)$, if $d,e,\varepsilon$ are small enough. Then, we have
\begin{align*}
\Real(A_{j,\mu_j}x^{k+\mu_j})&=-\Imag(A_{j,\mu_j})|x|^{k+\mu_j}\sin((k+\mu_j)\arg x)\\
&\le -\Imag(A_{j,\mu_j})|x|^{k+\mu_j}\sin((k+\mu_j)d'|x|^{r_j})
\end{align*}
so $\Real(A_{j,\mu_j}x^{k+\mu_j})\le -c_j|x|^{k+r_j+\mu_j}$, with $c_j=\Imag(A_{j,\mu_j})(k+\mu_j)d'/2$, if $d,e,\varepsilon>0$ are small enough. This proves (i).  
\end{proof}

Up to a linear change of coordinates $\zz^m\mapsto P\zz^m$, we can assume that the nonzero terms of the nilpotent part $N_2$ of the matrix $C_2$ are all equal to $c/2$, where $c>0$ is the constant appearing in Lemma~\ref{lem:nodaldomain}.

\subsection*{Existence of the stable manifold}
We prove here that for every $m\ge p+2$ there exists a stable manifold $\mathcal{S}_m$ of dimension $s$ given by a graph $\zz^m=\varphi_m(x,\ww^m)$ over a domain of the form
$$\basin=\left\{(x,\ww^m)\in\C\times\C^{s-1}: x\in\petal, \|\ww^m\|<|x|^{m-1}\right\}$$
where $d,e,\varepsilon>0$. As we will see, these stable manifolds are essentially the same for different values of $m$. In the proof, we can see that the contact of $\mathcal{S}_m$ with $\G$ increases with $m$. This will be key in the proof of asymptoticity of the orbits inside each $\mathcal{S}_m$.

We consider the vector spaces $\mathcal{C}(\basin,\C^{n-s})$  and $\mathcal{O}(\basin,\C^{n-s})$ of continuous and holomorphic maps respectively 
 from $\basin$ to $\C^{n-s}$ with the compact-open topology. Recall that 
since $\basin$ is a locally compact second countable space and  $\C^{n-s}$ is a complete metric space, $\mathcal{C}(\basin,\C^{n-s})$
is complete metrizable \cite[p. 272]{Dugundji}.
 
We will use the following result, which is an application of Schauder-Tychonoff theorem and is stated in \cite[p. 15]{Huk-K-M}. We include a proof for the sake of completeness.

\begin{proposition}\label{pro:fixedpointtheorem}
Given a continuous function $L:\basin\to\R_{\ge0}$, the set
$$\mathcal{H}_L=\left\{\varphi\in\mathcal{O}(\basin,\C^{n-s}): ||\varphi(x)||\le L(x) \text{ for all }(x,\ww^m)\in\basin\right\}$$
has the fixed point property; that is, every continuous map $T:\mathcal{H}_L\to\mathcal{H}_L$ has a fixed point.
\end{proposition}
\begin{proof}
The space $\mathcal{C}(\basin,\C^{n-s})$ is locally convex for the compact-open topology and the set $\mathcal{H}_L$ is clearly convex and closed. Moreover, by Montel theorem $\mathcal{H}_L$ is sequentially compact, and hence compact since $\mathcal{C}(\basin,\C^{n-s})$ is metrizable. By Schauder-Tychonoff theorem (see \cite{Dun-S}), every compact convex subset of a locally convex linear topological space has the fixed point property and this ends the proof.
\end{proof}

We define 
$$\spaceH=\left\{\varphi\in\mathcal{O}(\basin,\C^{n-s}): \left\|\varphi(x,\ww^m)\right\|\le |x|^{m-1} \text{ for all }(x,\ww^m)\in\basin\right\}.$$
The stable manifold $\mathcal{S}_m$ will be given by the graph of a fixed point $\varphi_m$ of a convenient continuous map $T:\spaceH\to\spaceH$.

Given $\varphi\in\spaceH$, we denote
$$f_\varphi(x,\ww^m)=f(x,\ww^m,\varphi(x,\ww^m)),\quad \overline{F}_{1,\varphi}(x,\ww^m)=\overline{F}_1(x,\ww^m,\varphi(x,\ww^m)).$$

\begin{proposition}\label{pro:basins}
If $d,e,\varepsilon>0$ are sufficiently small, then for all $\varphi\in\spaceH$ and $(x,\ww^m)\in\basin$ we have that
$$(f_{\varphi}(x,\ww^m),\overline{F}_{1,\varphi}(x,\ww^m))\in\basin.$$
\end{proposition}
\begin{proof}
Since $r<k+p$, we can argue as in \cite[Lemma 5.5]{Lop-R-R-S}, and we obtain that $$f_\varphi(\basin)\subset \petal$$ if $d,e,\varepsilon$ are sufficiently small.
Now, if $\varphi\in\spaceH$ we have by Lemma~\ref{lem:nodaldomain} that
\begin{align*}
\frac{\left\|\overline{F}_{1,\varphi}(x,\ww^m)\right\|}{\left|f_\varphi(x,\ww^m)\right|^{m-1}} & =
\frac{\left\|\ww^m\left(\exp(x^k \ol D_1(x))+O(x^{k+p})\right) + O(x^{k+p+m})\right\|}{\left|x-x^{k+p+1}+ O(x^{2k+2p+1})\right|^{m-1}}\cr
&\leq \frac{\left\|\ww^m\right\|}{\left|x\right|^{m-1}}\left\|\exp (x^k \ol D_1(x))+O(x^{k+p})
\right\||1+ O(x^{k+p})|+\|O(x^{k+p+1})\|\cr
&\leq \frac{\left\|\ww^m\right\|}{\left|x\right|^{m-1}}\left(1-c|x|^{k+t} + \|O(x^{k+t+1})\|
\right) |1+ O(x^{k+p})| + \|O(x^{k+p+1})\| \cr
& < 1-c|x|^{k+t}+\|O(x^{k+t+1})\| < 1
\end{align*}
for all $(x,\ww^m)\in\basin$, if $d,e,\varepsilon$ are sufficiently small.
\end{proof}

We consider $0<\varepsilon<1$ and fix $d,e>0$ small enough so that Lemma~\ref{lem:nodaldomain} and Proposition~\ref{pro:basins} hold. Given $\varphi\in\spaceH$ and $(x_0,\ww^m_0)\in\basin$, we denote
$$(x_j,\ww^m_j)=\left(f_\varphi(x_{j-1},\ww^m_{j-1}),\overline{F}_{1,\varphi}(x_{j-1},\ww^m_{j-1})\right), \; j\ge 1.$$
As in the classical one-dimensional case, there exists a constant $K\ge1$ such that
\begin{equation}\label{equ:dynamics1}
\lim_{j\to\infty} (k+p)jx_j^{k+p}=1 \quad \text{ and } \quad |x_j|^{k+p}\le K\frac{|x_0|^{k+p}}{1+(k+p)j|x_0|^{k+p}}
\end{equation}
for all $(x_0, \ww^m_0)\in\basin$ and all $j\in\N$, so in particular $(x_j,\ww^m_j)\to0$ when $j\to\infty$. Therefore, $\varphi_m\in\spaceH$ is a solution of the equation
\begin{equation}\label{equ:invariance}
\varphi(f_\varphi(x,\ww^m),\overline{F}_{1,\varphi}(x,\ww^m))=\overline{F}_2(x,\ww^m,\varphi(x,\ww^m))
\end{equation}
if and only if the set
$$\mathcal{S}_m=\left\{(x,\ww^m,\varphi_m(x,\ww^m)):(x,\ww^m)\in\basin\right\}$$
is a stable manifold of $F$.

Set $\rho=0$ when $k\ge1$, and $\rho = b - (p+1)/2$ when $k=0$ (recall that $b$ is the coefficient of $x^{2p+1}$
in $f(x,0)$). We define

$$E(x)=\exp\left(-\int\frac{\ol D_2(x)+x^pC_2}{x^{p+1}(1-\rho x^p)}dx\right),$$
where the integral is a notation for a primitive of   $x^{-(p+1)}  (1-\rho x^p)^{-1} (\ol D_2(x)+x^pC_2)$.

\begin{lemma}\label{lem:EE-1}
For any $(x,\yy^m,\zz^m)\in \basin\times\{\zz^m\in\C^{n-s}:\|\zz^m\|\le|x|^{m-1}\}$ with $\varepsilon$ sufficiently small, we have
$$E(x)E(f(x,\ww^m,\zz^m))^{-1}=\exp(-x^k(\ol D_2(x)+x^pC_2))+O(x^{k+p+1}).$$
\end{lemma}
\begin{proof} We argue as in \cite[Lemma 3.7]{Lop-S}. Observe that, since $\ol D_2(x)$ is diagonal and commutes with $C_2$, $E(x)$ is a fundamental solution of the linear system $x^{p+1}Y'=-B(x)Y$, where $B(x)=\left(\ol D_2(x)+x^pC_2\right)\left(1-\rho x^p\right)^{-1}$. Put
$\Omega(x,z)=E(x+x^{p+1}z)$. If we fix $x$ and consider $\Omega$ as a function of $z$ then it satisfies the (regular) system
$$\frac{\partial\Omega}{\partial z}=\frac{-B(x+x^{p+1}z)}{(1+x^{p}z)^{p+1}}\Omega(x,z).$$
On the other hand, we have $\Omega(x,0)=E(x)$ and thus
$$\Omega(x,z)=\exp\left(-\int_0^z\frac{B(x+x^{p+1}u)}{(1+x^{p}u)^{p+1}}du\right)E(x)$$
(using again that $\ol D_2(x)$ is diagonal and commutes with $C_2$). Hence
$$E(x)E(x+x^{p+1}z)^{-1}=E(x)\Omega(x,z)^{-1}=\exp\left(\int_0^z\frac{B(x+x^{p+1}u)}{(1+x^{p}u)^{p+1}}du\right).$$
The integrand in the equation above is an analytic function of $(x,u)$ and may be written as
$$\frac{B(x+x^{p+1}u)}{(1+x^{p}u)^{p+1}}=B(x)-(p+1)x^pB(x)u+O(x^{p+1}u, x^{2p}u^2).$$
Integrating, we obtain, for any $z$ sufficiently small,
$$E(x)E(x+x^{p+1}z)^{-1}=\exp\!\left(B(x)\Bigl(z-\frac{p+1}{2}x^pz^2\Bigr)\right)+x^{p+1}z^2\Lambda(x,z)+x^{2p}z^3\Theta(x,z),$$
where $\Lambda$ and $\Theta$ are analytic at the origin. The result follows using the expression of $f(x,\yy^m)$ and taking into account that $(x,\yy^m)\in\basin\times\{\|\zz^m\|\le|x|^{m-1}\}$, and thus $\|\yy^m\|\le|x|^{p+1}$, since $m\ge p+2$.
\end{proof}

\begin{lemma}\label{lem:sum-of-xj}
If $\varepsilon>0$ is small enough and, given $\varphi\in \spaceH$, we put $x_j=f_\varphi(x_{j-1},\ww^m_{j-1})$ and $\ww_j^m=\ol F_{1,\varphi}(x_{j-1},\ww^m_{j-1})$
for any $j\ge1$, then:
\begin{enumerate}[i)]
\item For any real number $l>k+p$ there exists a constant $K_l>0$ such that for any $(x_0,\ww_0)\in\basin$ and any $\varphi\in\spaceH$ we have
$$\sum_{j\geq 0}|x_j|^l\leq K_l|x_0|^{l-k-p}.$$
\item For any $(x_0,\ww_0)\in\basin$ and any $\varphi\in\spaceH$, we have $\|E(x_0)E(x_j)^{-1}\|\leq 1$ for every $j\ge0$.
\end{enumerate}
\end{lemma}
\begin{proof}
Part (i) follows from equation \eqref{equ:dynamics1}, as in \cite[Corollary 4.3]{Hak}.
To prove part (ii), observe that by Lemma~\ref{lem:EE-1}
$$E(x_0)E(x_1)^{-1}=\exp\left(-x_0^k(\ol D_2(x_0)+x_0^pC_2)\right)+\theta_{\varphi}(x_0,\ww^m_0),$$
where $\|\theta_\varphi(x_0,\ww^m_0)\|\le M_1|x_0|^{k+p+1}$ for any $(x_0,\ww^m_0)\in\basin$ and any $\varphi\in\spaceH$, with some $M_1>0$ independent of $\varphi$. We have that
$$\exp\left(-x^k(\ol D_2(x)+x^pC_2)\right)=\mathscr{D}\exp(-x^{k+p}N_2)=\mathscr{D}\left[I-x^{k+p}N_2+O(x^{k+p+1})\right],$$
where
$$\mathscr{D}=\diag\left(\exp\left(-x^kd_{s+1}(x)-x^{k+p}A_{s+1,p}\right),...,\exp\left(-x^kd_n(x)-x^{k+p}A_{n,p}\right)\right).$$
Then, using Lemma~\ref{lem:nodaldomain} and the fact that all the nonzero terms of $N_2$ are equal to $c/2$, we obtain
$$\left\|E(x_0)E(x_1)^{-1}\right\|\le 1-(c-c/2)|x_0|^{k+p}+M_2|x_0|^{k+p+1}\le1$$
for all $(x_0,\ww^m_0)\in\basin$ if $\varepsilon>0$ is sufficiently small. We obtain the result writing $E(x_0)E(x_j)^{-1}=\prod_{l=0}^{j-1} E(x_{l})E(x_{l+1})^{-1}$.
\end{proof}

Define
$$H(x,\ww^m,\zz^m)=\zz^m-E(x)E(f(x,\ww^m,\zz^m))^{-1}\overline{F}_2(x,\ww^m,\zz^m).$$
Using Lemma~\ref{lem:EE-1} we get
\begin{equation}\label{eq:H2}
H(x,\ww^m,\zz^m)=O(x^{k+p+1}\|\yy^m\|,x^{k+p+m}).
\end{equation}

\begin{proposition}\label{pro:fixedpoint}
If $\varepsilon>0$ is sufficiently small and we denote $x_j=f_\varphi(x_{j-1},\ww^m_{j-1})$ and $\ww^m_j=\overline{F}_{1,\varphi}(x_{j-1},\ww^m_{j-1})$ for any $j\ge1$, where $(x_0,\ww^m_0)\in\basin$ and $\varphi\in\spaceH$, then the series
$$T\varphi(x_0,\ww^m_0)=\sum\limits_{j\ge0}E(x_0)E(x_j)^{-1}H(x_j,\ww^m_j,\varphi(x_j,\ww^m_j))$$
is normally convergent and defines a map $T:\varphi\mapsto
T\varphi$ from $\spaceH$ to itself which is continuous for the topology of the uniform convergence. Moreover, $\varphi\in\spaceH$
is a fixed point of $T$ if and only if the set $\{(x,\ww^m,\varphi(x,\ww^m):(x,\ww^m)\in\basin\}$ is a stable manifold of $F$.
\end{proposition}

\begin{proof}
If $\varphi\in\spaceH$ and $(x_0,\ww^m_0)\in\basin$, then by equation~\eqref{eq:H2} we have that $\left\|H(x_j,\ww^m_j,\varphi(x_j,\ww^m_j))\right\|\le M|x_j|^{k+p+m}$ for some $M>0$, so by Lemma~\ref{lem:sum-of-xj} we get
$$\left\|T\varphi(x_0,\ww^m_0)\right\|\le M\sum\limits_{j\ge0}|x_j|^{k+p+m}$$
and the series is normally convergent by Lemma~\ref{lem:sum-of-xj}. Moreover we have that
$\left\|T\varphi(x,\ww^m)\right\|\le MK_{k+p+m}|x|^m\le |x|^{m-1}$ if $\varepsilon>0$ is sufficiently small, so $T\varphi\in\spaceH$. Continuity of $ T$ follows from the uniform convergence of the series with respect to $\vp$. Finally, we rewrite
\begin{align*}
	T\varphi(x_0,\ww^m_0)&=E(x_0)\sum_{j\ge0}\left[E(x_j)^{-1}\varphi(x_j,\ww^m_j)-E(x_{j+1})^{-1}\ol F_2\left(x_j,\ww^m_j,\varphi(x_j,\ww^m_j)\right)\right]\\
	&=\varphi(x_0,\ww^m_0)-E(x_0)E(x_1)^{-1}\left[\ol F_2\left(x_0,\ww_0^m,\varphi(x_0,\ww^m_0)\right)-T\varphi(x_1,\ww^m_1)\right].
\end{align*}
From these two equalities it follows that $\varphi$ is a fixed point of $T$ if and only if $\varphi$ satisfies the invariance equation~\eqref{equ:invariance}, i.e. if and only if the set $\{(x,\ww^m,\varphi(x,\ww^m):(x,\ww^m)\in\basin\}$ is a stable manifold of $F$.
\end{proof}

By Proposition~\ref{pro:fixedpointtheorem}, $T$ has a fixed point $\varphi_m\in\spaceH$. Hence, by Proposition~\ref{pro:fixedpoint}, the set $$\mathcal{S}_m=\{(x,\ww^m,\varphi_m(x,\ww^m)):(x,\ww^m)\in \basin\}$$ is a stable manifold of $F$.

\subsection*{Stable manifold as a base of asymptotic convergence}

Let us show that every orbit $\{(x_j,\yy^m_j)\}$ of $F$ which is asymptotic to $\G$ and such that $\{x_j\}$ has $\R^+$ as tangent direction is eventually contained in $\mathcal{S}_m$. Since the order of contact of $\G$ with the $x$-axis is at least $p+m$, any orbit $\{(x_j,\yy^m_j)\}$ asymptotic to $\G$ satisfies $\|\yy^m_j\|<|x_j|^{p+m-1}$ if $j$ is sufficiently large. Therefore, the result is a consequence of the following lemma.

\begin{lemma}\label{lem:base-asymp}
Let $\{(x_j,\ww^m_j,\zz^m_j)\}$ be a stable orbit of $F$ such that $\{x_j\}$ has $\R^+$ as tangent direction and such that $\|\ww^m_j\|<|x_j|^{m-1}$ for all $j$ sufficiently big. Then $(x_j,\ww^m_j,\zz^m_j)\in \mathcal{S}_m$ for all $j$ sufficiently big.
\end{lemma}
\begin{proof}
Since $\{x_j\}$ has $\R^+$ as tangent direction, we obtain, arguing exactly as in \cite[Lemma 5.8]{Lop-R-R-S}, that $x_j\in \petal$ if $j$ is sufficiently big and hence $(x_j,\ww^m_j)\in \basin$ for all $j\ge j_0$. Consider the change of coordinates $\zz^m\mapsto\zz^m-\varphi_m(x,\ww^m)$, valid on  $\basin\times\C^{n-s}$. In the new coordinates the stable manifold $\mathcal{S}_m$ is given by $\zz^m=0$ and hence $F$ is written as
\begin{align*}
	f(x,\yy^m)&=x-x^{k+p+1}+  b x^{2k+2p+1} + O(x^{2k+2p+2})  \\[2pt]
	\ol F_1(x,\yy^m)&=\exp\left(x^k\left(\ol D_1(x)+x^pC_1\right)\right)\ww^m+O(x^{k+p+1}\|\yy^m\|,x^{k+p+m})\\[2pt]
	\ol F_2(x,\yy^m)&=\exp\left(x^k\left(\ol D_2(x)+x^pC_2\right)\right)\zz^m+O(x^{k+p+1}\|\zz^m\|).
\end{align*}
By Lemma~\ref{lem:nodaldomain} we obtain
$$\|\ol F_2(x_j,\yy_j^m)\|\ge \left(1+c|x_j|^{k+p}+O(x_j^{k+p+1})\right)\|\zz_j^m\|\ge\|\zz_j^m\|$$
for all $j\ge j_0$, so we conclude that if $\zz^m_{j_0}\neq0$ the orbit $\{(x_j,\ww^m_j,\zz^m_j)\}$ cannot converge to the origin. Therefore, $(x_j,\ww^m_j,\zz^m_j)\in \mathcal{S}_m$ for any $j\ge j_0$.
\end{proof}

\begin{remark}\label{rk:unicity-vpm}
Note that Lemma~\ref{lem:base-asymp} also implies that $\varphi_m$ is actually the unique fixed point of $T$ in $\spaceH$.
\end{remark}

\subsection*{Asymptoticity of the orbits}

To finish the proof of Theorem~\ref{th:stablemanifold} it only remains to prove that every orbit in $\mathcal{S}_m$ is asymptotic to $\G$. Observe that, since the order of contact of $\G$ with the $x$-axis is at least $p+m$ and the order of contact of $\mathcal{S}_m$ with the $x$-axis is at least $m-1$, the order of contact of $\mathcal{S}_m$ with $\G$ is at least $m-1$. We will show that every orbit $\{(x_j,\yy^m_j)\}\subset\mathcal{S}_m$, which has order of contact at least $m-1$ with $\G$, is eventually contained in $\mathcal{S}_{m+1}$, and therefore its order of contact with $\G$ is at least $m$. Applying this argument recursively, we conclude that every orbit in $\mathcal{S}_m$ is asymptotic to $\G$.

\begin{lemma}\label{lem:asympt}
Fix $\varepsilon,d,e>0$ sufficiently small. Let $\{(x_j,\ww^m_j,\zz^m_j)\}$ be a stable orbit of $F$ such that $x_j\in\petal$ and $\|\zz^m_j\|<|x_j|^{m-1}$ for all $j$. Then $\|\ww^m_j\|<\frac 12|x_j|^m$ for all $j$ sufficiently large.
\end{lemma}
\begin{proof}
By Lemma~\ref{lem:nodaldomain} we have
\begin{align*}
	\frac{\|\ww^m_{j+1}\|}{|x_{j+1}|^m} & =
	\frac{\bigl\|\ww^m_j\bigl(\exp(x_j^k \ol D_1(x_j))+O(x_j^{k+p})\bigr) + O(x_j^{k+p+m})\bigr\|}{\bigl|x_j-x_j^{k+p+1}+ O(x_j^{2k+2p+1})\bigr|^m}\cr
	&\leq \frac{\|\ww_j^m\|}{|x_j|^m} \left(1-c|x_j|^{k+t}+O(x_j^{k+t+1})\right)+\|O(x_j^{k+p})\|
\end{align*}
for all $j$. This implies, since $t<p$, that if $\|\ww^m_j\|<\frac 12|x_j|^m$ then $\|\ww^m_{j+1}\|<\frac 12|x_{j+1}|^m$. Therefore, to prove the lemma it suffices to show that $\|\ww^m_j\|<\frac 12|x_j|^m$ for some $j$. Suppose this is not the case, so $\|\ww^m_j\|\ge \frac 12|x_j|^m$ for all $j\ge0$. Then
$$\frac{\|\ww^m_{j+1}\|}{|x_{j+1}|^m}\le \frac{\|\ww_j^m\|}{|x_j|^m} \left(1-c|x_j|^{k+t}+O(x_j^{k+t+1})\right)$$
for all $j\ge0$, so we obtain that
$$\frac{\|\ww^m_{j+1}\|}{|x_{j+1}|^m} \le \frac{\|\ww_0^m\|}{|x_0|^m}\prod_{l=0}^{j}\left(1-c|x_l|^{k+t}+O(x_l^{k+t+1})\right).$$
Since $\lim_{j\to\infty} (k+p)jx_j^{k+p}=1$ and $t<p$, the product above converges to 0 when $j\to\infty$, contradicting the fact that $\|\ww^m_j\|\ge \frac 12|x_j|^m$ for all $j$.
\end{proof}

Consider an orbit $\{(x_j,\ww^m_j,\zz^m_j)\}\subset\mathcal{S}_m$ and consider the coordinates $(x,\yy^{m+1})=(x,\ww^{m+1},\zz^{m+1})$ satisfying  $\yy^{m+1}=\yy^m-\left(J_{p+m}\overline\g(x)-J_{p+m-1}\overline\g(x)\right)$, where $\g(s)=(s,\overline\g(s))$ is a parametrization of $\G$. By Lemma~\ref{lem:asympt}, $\|\ww^m_j\|<\frac 12|x_j|^m$ for all $j$ sufficiently large, so $\|\ww^{m+1}_j\|<\frac 12|x_j|^m+M|x_j|^{p+m}$ for some $M>0$ and for all $j$ sufficiently large. Then, we get that $\|\ww^{m+1}_j\|<|x_j|^m$ for all $j$ sufficiently large, since we can assume that  $p\ge1$ (otherwise the variables $\ww^m$ do not appear). Therefore, by Lemma~\ref{lem:base-asymp}, $(x_j,\ww^{m+1}_j,\zz^{m+1}_j)\in \mathcal{S}_{m+1}$ if $j$ is big enough.
This shows that every orbit in $\mathcal{S}_m$ is asymptotic to $\G$.

This ends the proof of Theorem~\ref{th:stablemanifold}.

\end{document}